\documentclass[10pt]{article}
\usepackage[utf8]{inputenc}
\usepackage[colorlinks=true]{hyperref}
\usepackage{fullpage}

\title{On $G$-crossed Frobenius $\star$-algebras and fusion rings associated with braided $G$-actions}
\author{Prashant Arote and Tanmay Deshpande}
\date{}

\usepackage{bbm}

\usepackage{graphicx}
\makeatletter
\@ifpackageloaded{hyperref}%
  {\newcommand{\mylabel}[2]
    {\protected@write\@auxout{}{\string\newlabel{#1}{{#2}{\thepage}%
      {\@currentlabelname}{\@currentHref}{}}}}}%
  {\newcommand{\mylabel}[2]
    {\protected@write\@auxout{}{\string\newlabel{#1}{{#2}{\thepage}}}}}
\makeatother

\usepackage{amsmath}
\usepackage{tikz-cd}
\usepackage[all,cmtip]{xy}
\usepackage{graphicx}
\usepackage{amsmath}
\usepackage{wasysym}
\usepackage{amssymb}
\usepackage{amsthm}
\usepackage{commath}
\usepackage{mathtools}
\usepackage{tikz-cd}
\usepackage[utf8]{inputenc}
\usepackage[nottoc]{tocbibind}
\usepackage{hyperref}

\newtheorem{theorem}{Theorem}[section]

\theoremstyle{definition}
\newtheorem{definition}[theorem]{Definition}
\newtheorem{prop}[theorem]{Proposition}

\newtheorem{lemma}[theorem]{Lemma}
\newtheorem{example}[theorem]{Example}

\newtheorem{remark}[theorem]{Remark}

\newtheorem{corollary}[theorem]{Corollary}

\newcommand{\s}{\star}
\newcommand{\A}{A}
\newcommand{\B}{\mathcal{B}}
\newcommand{\C}{\mathbb{C}}
\newcommand{\D}{\mathcal{D}}
\newcommand{\ZG}{\mathbb{Z}[G]}
\newcommand{\Z}{\mathbb{Z}}
\newcommand{\M}{\mathcal{M}}

\newcommand{\Rep}{\Sim(\A_{1})}
\newcommand{\Sim}{\operatorname{Sim}}
\newcommand{\tr}{\operatorname{trace}}
\newcommand{\End}{\operatorname{End}}
\newcommand{\Pic}{\operatorname{Pic}}
\newcommand{\Hom}{\operatorname{Hom}}
\newcommand{\CoInd}{\operatorname{CoInd}}
\newcommand{\Frob}{\operatorname{Frob}}
\newcommand{\Supp}{\operatorname{Supp}}
\newcommand{\EqMod}{\operatorname{EqMod}}
\newcommand{\EqBr}{\operatorname{EqBr}}
\newcommand{\trace}{\operatorname{tr}}

\begin{document}
\maketitle

\begin{abstract}
For a finite group $G$, Turaev introduced the notion of a  braided $G$-crossed fusion category. 
The classification of braided $G$-crossed extensions of braided fusion categories was studied by Etingof, Nikshych and Ostrik in terms of certain group cohomological data.
 In this paper we will define the notion of a $G$-crossed  Frobenius $\star$-algebra and give a classification  of
 (strict) $G$-crossed extensions of a commutative Frobenius $\star$-algebra $R$ equipped with a given action of $G$, in terms of the second group cohomology $H^2(G,R^\times)$. Now suppose that $\B$ is a non-degenerate braided fusion category equipped with a braided action of a finite group $G$. We will see that the associated $G$-graded fusion ring is in fact a (strict) $G$-crossed Frobenius $\star$-algebra. We will describe this $G$-crossed fusion ring in terms of the classification of braided $G$-actions by Etingof, Nikshych, Ostrik and derive a Verlinde formula to compute its fusion coefficients. 
\end{abstract}

\section{Introduction}
Let  $G$ be a  finite group.
The goal of this paper is to define and study the notion of a $G$-crossed Frobenius $\s$-algebra and classify the $G$-crossed extensions of a commutative Frobenius $\s$-algebra in terms of  some cohomological data. The next goal is to use these results to describe the $G$-crossed fusion ring associated with a braided $G$ action on a non-degenerate braided fusion category. All algebras considered in this paper are over $\C$, though for some definitions and results it would also be possible to work over slightly more general commutative rings equipped with an involution. However most of the results make use of the properties of $\mathbb{C}$. We will now give some motivation behind the results of this paper.

Let us begin with the motivation behind the notion of a Frobenius $\star$-algebra (see Definition \ref{def:Frob star algebra}). Let $\mathcal{C}$ be fusion category over $\C$, namely a $\C$-linear finite  semisimple abelian rigid monoidal category whose unit object $\mathbbm{1}$ is simple. 
For more details about fusion and multi-fusion categories we refer to \cite{Onfusioncategories}. Then the complexified Grothendieck ring $K(\mathcal{C})$ of $\mathcal{C}$ has the structure of a Frobenius $\star$-algebra, with the $\star$-anti-involution being induced by the duality in $\mathcal{C}$, extended semi-linearly to $K(\mathcal{C})$. In particular, for any finite group $G$, the group algebra $\C[G]$ is naturally a Frobenius $\star$-algebra. More generally, any (complexified) based ring, in the sense of \cite{Lus:Arcata}, is naturally a Frobenius $\star$-algebra. 

Let us now study the motivation for $G$-graded and $G$-crossed Frobenius $\star$-algebras. A $G$-extension of a fusion category $\mathcal{C}$ is a $G$-graded fusion category $\mathcal{D}=\bigoplus\limits_{g\in G}\mathcal{C}_{g}$ whose trivial component $\mathcal{C}_1$ is equivalent to $\mathcal{C}$.
The  $G$-extensions of fusion categories were studied
by Etingof, Nikshych and Ostrik in \cite{FusionCA}.
The notion of a braided
$G$-crossed fusion category was introduced by  Turaev, see \cite{Turaev1999HomotopyFT},\cite{crossedgroupcategories}.
Roughly speaking, it is a $G$-graded fusion category $\D=\bigoplus\limits_{g\in G}\mathcal{C}_{g}$, equipped with a monoidal action of $G$ such that $g(\mathcal{C}_{h})=\mathcal{C}_{ghg^{-1}}$ and a family of natural isomorphisms 
$$c_{X,Y}:X\otimes Y\rightarrow g(Y)\otimes X\ \ \ \ \mbox{for any }g\in G,\ X\in\mathcal{C}_{g},\ Y\in \D$$
called $G$-braiding isomorphisms.
The above action and $G$-braiding are required to satisfy certain compatibility conditions.
In particular, its trivial component $\mathcal{C}_{1}$ is a braided fusion category equipped with a braided action of $G$.
A $G$-crossed extension of braided fusion category $\mathcal{B}$ is a $G$-crossed fusion category $\D$ whose trivial component is equivalent to $\B$.
We refer to \cite[Section 4.4.3]{DGNO} for a detailed discussion of braided $G$-crossed categories.
The classification of $G$-crossed extensions of a (non-degenerate) braided fusion category was carried out in \cite[Theorem 7.12]{FusionCA}.

Let $K(\mathcal{D})$ denote the complexified Grothendieck ring of a braided $G$-crossed fusion category $\mathcal{D}$.
Then $K(\D)$ is a $G$-graded Frobenius $\star$-algebra with an action of $G$.
This action of $G$ on $K(\D)$ satisfies certain conditions induced by the $G$-braiding.
Moreover, if $\mathcal{D}$ is a $G$-crossed extension of braided fusion category $\mathcal{B}$ then the trivial graded component of $K(\mathcal{D})$ is equal to $K(\mathcal{B})$.
The complexified Grothendieck ring $K(\D)$ is an example of $G$-crossed Frobenius $\star$-algebra (see \textsection\ref{sec:motivation}).
One of the motivation behind the study of  $G$-crossed Frobenius $\star$-algebra is to study the  fusion ring $K(\D)$.

We will define $G$-crossed Frobenius $\star$-algebra by abstracting the properties satisfied by  complexified Grothendieck ring of a $G$-crossed braided fusion category. 
Roughly, \textit{a $G$-crossed Frobenius $\star$-algebra} is a $G$-graded Frobenius $\star$-algebra equipped with  an action of $G$ which satisfies certain conditions (see Definition \ref{def: crossed algebra}). In particular its identity component is a commutative Frobenius $\star$-algebra equipped with a $G$-action.
A $G$-crossed extension of a commutative Frobenius $\star$-algebra $R$ is a $G$-crossed Frobenius $\star$-algebra whose trivial component is isomorphic to $R$. One of our main goals in this paper is to classify such $G$-crossed extensions.

We now briefly describe the organization and main results of this paper. In \textsection \ref{sec:Frob star} we will recall the notion of a Frobenius $\star$-algebra and that they are semi-simple algebras.
A basic example of a Frobenius $\s$-algebra is a matrix algebra $M_{n}(\C)$.
We  prove that there is a bijective correspondence  between the set of isomorphism classes of commutative Frobenius $\star$-algebras and the set of tuples of positive real numbers upto permutation.
In  \textsection\ref{sec:graded} we will recall the definition of a group graded ring and of a $G$-graded Frobenius $\star$-algebra.

In \textsection \ref{sec:crossed algebra} we will define the notion of a $G$-crossed Frobenius $\star$-algebra and that of a  strict $G$-crossed Frobenius $\star$-algebra.
In \textsection \ref{sec:motivation} we will see important examples of (strict) $G$-crossed Frobenius $\star$-algebras that arise from braided $G$-actions on non-degenerate braided fusion categories. We will prove some results about the structure and classification of strict $G$-crossed Frobenius $\star$-algebras (see Prop. \ref{G-crosed str}).
In particular, we will prove the following main result:
\begin{theorem}\label{main result}
Let $G$ be a finite group and let $R$ be a commutative Frobenius $\star$-algebra equipped with an action of $G$ by Frobenius $\star$-algebra automorphisms.
Then there is a bijective correspondence between the set of isomorphism classes of  strict $G$-crossed  extensions of $R$, denoted by $\Frob(G,R)$, and the second group cohomology $H^{2}(G,R^{\times})$ \emph{i.e.}
$$\Frob(G,R)\longleftrightarrow H^{2}(G,R^{\times}).$$
\end{theorem}

We will also state and prove an analogue of the Verlinde formula in case of a strict $G$-crossed Frobenius $\star$-algebra (see Corollary \ref{Verlinde formula:any genus}).
The twisted categorical Verlinde formula in case of braided $G$-crossed categories was proved in \cite{Deshpande2019CrossedMC}. We will then derive a slight generalization of this result from \cite{Deshpande2019CrossedMC} as described below. We will also relate the above result to the classification of braided $G$-actions and braided $G$-crossed fusion categories of \cite{FusionCA}. 

Let $\B$ be a non-degenerate braided fusion category and let $c:G\to{}\EqBr(\B)\cong \Pic(\B)$ be a group homomorphism from $G$ to the group $\EqBr(\B)$ of braided auto-equivalences of $\B$ up to natural isomorphisms which is isomorphic to the group $\Pic(\B)$ of equivalence classes of invertible $\B$-module categories. We refer to \cite{FusionCA} for more details about these and related results. Note that such a group homomorphism induces an action of $G$ by Frobenius $\star$-algebra automorphisms on $K(\B)$. Let $\mathcal{O}_{\B}^{\times}$ denote the abelian group of isomorphism classes of invertible (in particular they will also be simple) objects in $\B$. Note that $c$ also induces an action of $G$ on the commutative group $\mathcal{O}^\times_\B$. We have the short exact sequence of $G$-modules
\[
1\to \mathcal{O}_{\B}^{\times}\to K(\B)^\times\to  K(\B)^\times/\mathcal{O}_{\B}^{\times}\to 1.
\]

Now given such a $c$, \cite{FusionCA} construct an element $T\in H^3(G,\mathcal{O}^\times_\B)$. The lifts of $c$ to a braided $G$ action on $\B$, namely morphisms of 1-groups $\underline{c}:G\to \underline{\EqBr}(\B)\cong \underline{\Pic}(\B)$ lifting $c$, are classified by a certain $H^2(G,\mathcal{O}^\times_\B)$-torsor $\mathcal{T}$ which is non-empty if and only if $T\in H^3(G,\mathcal{O}^\times_\B)$ vanishes. Given such a $c$, we will construct an element $t\in H^2(G,K(\B)^\times/\mathcal{O}_{\B}^{\times})$ which is mapped to $T\in H^3(G,\mathcal{O}^\times_\B)$ under the connecting homomorphism (see Corollary \ref{associativity}) coming from the above short exact sequence. When the obstruction $T$ vanishes, we construct a mapping 
\begin{equation}\label{eq:Phi}\Phi:\mathcal{T}\to H^2(G,K(\B)^\times).\end{equation}

Let $\underline{c}:G\to \underline{\EqBr}(\B)\cong \underline{\Pic}(\B)$ define an action of $G$ by braided auto-equivalences on $\B$ lifting $c$. Such a braided $G$-action corresponds to an element $[\underline{c}]\in \mathcal{T}$. Using the results of \cite{FusionCA}, we can define an associated strict $G$-crossed Frobenius $\star$-extension of $K(\B)$, which we call the fusion algebra associated with the braided $G$-action, see Proposition \ref{Thm:pic crossed algebra}. We will see that the element (given by Theorem \ref{main result}) of $H^2(G,K(\B)^\times)$ corresponding to this strict $G$-crossed Frobenius $\star$-extension is given by $\Phi([\underline{c}])$ (see also Remark \ref{rk:cohclass}). We will derive a Verlinde formula to compute the fusion rules in the above $G$-crossed Frobenius $\star$-algebra. We will also prove an analogue of this result in case $\B$ is a modular fusion category (i.e. $\B$ is also equipped with a ribbon structure, see \cite{DGNO}) and the $G$-action is modular, and express the fusion coefficients in terms of the associated crossed S-matrices. We refer to Corollaries \ref{coro: twisted Verlinde formula} and \ref{coro:twisted Verlinde ribbon} for these categorical twisted Verlinde formulae.

\section{Frobenius $\s$-algebra}\label{sec:Frob star}
We will begin by recalling the definition of Frobenius algebras,  Frobenius $\s$-algebras and some of their properties.
For more details see \cite{Twisted}.
\begin{definition}[Frobenius Algebra]
 (i) A Frobenius algebra $\A$ is a finite dimensional associative unital $\C$-algebra equipped with a linear functional $\lambda:\A\rightarrow\C$ such that the bilinear form on $A$ defined by $(a,b)=\lambda(ab)$  is nondegenerate.\\
 (ii) Moreover if $\lambda(ab)=\lambda(ba),\ \forall\ a,\ b\in \A$, then $\A$ is called as symmetric Frobenius algebra. (A linear functional $\lambda:A\to \C$ on an algebra satisfying this condition is said to be a class functional.)
\end{definition}
\begin{remark}
There is an equivalent definition of a Frobenius algebra: A Frobenius algebra $\A$ is a finite-dimensional, unital, associative $\C$-algebra  equipped with a left $\A$-module isomorphism $\Theta:\A\rightarrow\A^{\vee}$. 

The equivalence between the two definitions is given by:
\begin{gather*}
(\A,\lambda)\ \rightsquigarrow\ \Theta(a)=\lambda(- \cdot a) \\
(\A,\phi)\ \rightsquigarrow\ \lambda (x)=\Theta (1)(x).
\end{gather*}
\end{remark}
\begin{definition}[Frobenius $\s$-algebra]\label{def:Frob star algebra}
A Frobenius $\s$-algebra $\A$ is a finite dimensional associative unital $\C$-algebra equipped with a {\emph{ class functional}} $\lambda:\A\rightarrow\C$ and a map $\s:\A\rightarrow\A$, such that the following  holds:

$\bullet\ (a+b)^{\s}=a^{\s}+b^{\s}\ \forall\ a,\ b\in \A,$

$\bullet \ (a^{\s})^{\s}=a\ \forall\ a\in \A,$

$\bullet\ (ab)^{\s}=b^{\s}a^{\s}\ \forall\ a,b\in \A,$

$\bullet\ (\alpha a)^{\s}=\overline{\alpha}a^{\s}\  \forall\ \alpha\in \C, a\in\A,$

$\bullet\ \lambda(a^\star)=\overline{\lambda(a)} \ \forall\  a\in A,$

$\bullet\ $The Hermitian form $\langle a,b\rangle:=\lambda(ab^\star)$ admits an orthonormal basis in $A$.
\end{definition}

\begin{remark}
Note that by the given conditions, $\overline{\lambda(ab^\star)}=\lambda((ab^\star)^\star)=\lambda(ba^\star)$. Hence $\langle\cdot,\cdot\rangle$ is indeed a Hermitian form on $A$. Since all our algebras are over $\C$, the existence of an orthonormal basis in the last condition is equivalent to the Hermitian form being positive definite.
\end{remark}
\begin{definition}
Let $(A_{1},\lambda_{1},\star_{1})$ and  $(A_{2},\lambda_{2},\star_{2})$  be two Frobenius $\star$-algebras.
An algebra homomorphism $f:A_{1}\rightarrow A_{2}$ is called as Frobenius $\star$-algebra homomorphism if $\lambda_{1}(a)=\lambda_{2}\circ f(a)$ and $f(a^{\star_{1}})=f(a)^{\star_{2}}$ for all $a\in A_{1}$.
We will denote the set of all Frobenius $\star$-algebra automorphisms of $(A,\lambda,\star)$ by $\mbox{Aut}^{\Frob}(A)$.
\end{definition}
\begin{example}
(1) $\A=M_{n}(\C)$ with linear functional $\lambda=$ trace and the $\s$-map is conjugate transpose \emph{i.e.} $X^{\s}=\overline{X}^{t}$. \\
(2) Let $G$ be a finite group. 
Then $\A=\C[G]$ with linear functional $\lambda$ defined by 
$\lambda(g)=$
$\begin{cases}
1\ \mbox{if}\ g=e ;\\
0\ \mbox{otherwise}.
\end{cases}$ 
The $\s$-map is defined by $g^{\s}=g^{-1}$ and extend to $\C[G]$ by conjugate linearity.
\end{example}
\begin{remark}
Let $R$ be a commutative ring with unity which is equipped with an involution $\overline{(\cdot)}:R\rightarrow R$.
One can also define the notion of  Frobenius algebra and Frobenius $\star$-algebra for $R$-algebras, but in this paper we stick to $\C$ with the involution being complex conjugation. 
\end{remark}

\begin{remark}\label{frob star:properties} 
Let $(\A,\lambda,\star)$ be a Frobenius $\s$-algebra. Then the following  holds:
\begin{enumerate}
\item $(\A,\lambda)$ is symmetric Frobenius algebra. 
\item $\A$ is a semisimple algebra (see \cite[Theorem 2.4]{Twisted}).
\item Let $\Sim(\A)$ denote the set of irreducible representations of $\A$.
Let $\{e_{M}:M\in\Sim(\A)\}$ be the set of primitive central orthogonal idempotents of $\A$ then $e_{M}^{\s}=e_{M}$, for all $M\in\Sim(\A)$ (see \cite[Lemma 2.5]{Twisted}).
\end{enumerate}

\end{remark}

Next, we will see the possible Frobenius $\s$-algebra structures on $\End(V)$ for a finite dimensional $\C$-vector space $V$.

\begin{prop}\label{matrix algebra}
Let $V$ be a  finite dimensional vector space over $\C$.
Let $(\End(V),\lambda,\star)$ be a Frobenius $\s$-algebra.
Then  $\lambda=\alpha\cdot \tr$, for some positive real number $\alpha$ and $\star:\End(V)\rightarrow\End(V)$ is an adjoint map corresponding to some positive definite Hermitian form on $V$.
\end{prop}
\begin{proof}
As $\lambda$ is a non-degenerate \emph{class functional} on $\End(V)$ it must be a non-zero multiple of the trace: $$\lambda(X)=\alpha\cdot\tr(X)\ \forall\  X\in \End(V)$$ 
for some $\alpha\in\C^\times$.
As  $Id^{\s}=Id$ in $\End(V)$, by positive definiteness we have 
$$0< \langle Id,Id\rangle=\lambda(Id\cdot Id^{\s})=\alpha\cdot\tr(Id)=\alpha\cdot n.$$
Thus $\lambda=\alpha\cdot\tr$ for some $\alpha\in \mathbb{R}^{+}$.

We know that $\End(V)=V^*\otimes V\cong V^{\oplus n}$ as a left $\End(V)$-module.
Fix a left $\End(V)$-module embedding  $f:V\hookrightarrow \End(V)$.
By the definition of a Frobenius $\s$-algebra, $\End(V)$ is equipped with a positive definite Hermitian form. Define a positive definite Hermitian form on $V$ using the embedding $f$ as:
$$\langle v,w\rangle_{f}=\lambda(f(v)f(w)^{\star})\ \ \ \forall \ v,w\in V.$$
Then by the definition of a Frobenius $\star$-algebra, $\langle\cdot,\cdot \rangle_{f}$ is a positive definite Hermitian form on $V$.
Let $T\in \End(V)$ be any endomorphism, then for $v,w\in V$, $$\langle Tv,w\rangle_{f}=\lambda(f(Tv)f(w)^{\star})=\lambda(Tf(v)f(w)^{\star})=\lambda(f(v)f(w)^{\star}T)$$ $$=\lambda(f(v)(T^{\star}f(w))^{\star})=\lambda(f(v)f(T^{\star}w)^{\star})=\langle v,T^{\star}w\rangle_{f}.$$
Thus $\star$ is an adjoint map corresponding to positive definite Hermitian form $\langle\cdot ,\cdot\rangle_{f}$.
This proves the proposition.
\end{proof}
\begin{theorem}\label{classificationfrobstar}
Let $A$ be any Frobenius $\star$-algebra then 
$$A\cong\bigoplus_{i=1}^{r} M_{n_{i}}(\C)$$
where each $M_{n_{i}}(\C)$ is a Frobenius $\star$-algebra with a linear functional $\lambda_i=\alpha_{i}\cdot\tr$ for some $\alpha_{i}\in\mathbb{R}^{+}$ and the $\star$-map is conjugate transpose.
\end{theorem}
\begin{proof}
  We know that any Frobenius $\star$-algebra is a semisimple algebra and the $\star$-map preserve the primitive central simple idempotents.
Then the statement follows from  the Prop. \ref{matrix algebra}.
 \end{proof}
\noindent We obtain the following from the above result:
 \begin{corollary}\label{Commutative Frob algebra}
There is a bijection between the set of isomorphism classes of $n$-dimensional commutative Frobenius $\s$-algebras and  the set of  $n$-tuples of positive real numbers up to permutation.
 \end{corollary}
 \begin{corollary}\label{auto-star commute}
  Let $A$ be a commutative Frobenius $\star$-algebra and let $f:A\rightarrow A$ be any algebra automorphism.
  Then $f$ commute with $\star$-map \emph{i.e.}
  $$f(a^{\star})=f(a)^{\star}\ \ \ \ \ \ \forall\ a\in A.$$
 \end{corollary}

\section{Group graded algebras}\label{sec:graded}
We begin by briefly recalling the notion of group graded algebras and some results about them. For more details we refer to \cite{Dade1986}.

\begin{definition}
A $G$-graded ring $\A$  is a ring with an internal direct sum decomposition:
$A = \oplus_{g \in G} A _{g}$ as (additive groups),
where the additive subgroups  $A_{g}$ of  $A$  satisfy:
$A_{g}A_{h}\subseteq A_{gh}, \ \forall\ g, h \in  G$.
\end{definition}
The above decomposition is called as a $G$-grading of $\A$, and its summands $\A_{g}$ are called as $g$-components of $\A$. 
The component of $\A$  corresponding to $1\in G$ is a subring of $\A$ and
 $\A$  is an unital bimodule over $\A_{1}$. 
 Also $\A_{g}$ is an unital bimodule over $\A_{1}$ for each $g \in G$.
 If $\A$  is a $G$-graded ring and $H$ is a subgroup of $G$\ then we naturally obtain a $H$-graded ring from $\A$, $\A_{H}=\oplus_{h\in H} \A_{h}.$

\begin{remark}
We will  use the same notation i.e. ``1'' to denote the identity element of a finite group $G$ and for the unity in an algebra $A$.
\end{remark}

\begin{definition}
 A $G$-graded module over a $G$-graded ring $A$ \ is an $A$ -module $M$ with an internal 
direct sum decomposition: $M =\sum_{g\in G} M_{g},$ (as $ A_{1}$-modules),
where $A_{1}$-modules $M_{g}$\ satisfy:
$A_{g}M_{\sigma}\subseteq M_{g\sigma},\ \forall
 \ g,\sigma\in G$.
\end{definition}

Suppose $M$\ is a $G$-graded module over a $G$-graded ring $A$ then we get a natural $H$-graded $A_{H}$-module $M_{H}=\sum_{h\in H}M_{h}$. 
A $G$-graded $A$-submodule $N$ of $M$ is an $A$-submodule which has $G$-grading with $g$-component,
$N_{g}=M_{g}\cap N,\ \ \forall\ g\in G$.

Let $M$ be a $G$-graded $A$-module.
We say that an $G$-graded $A$-submodule $U$ of $M$ is \textbf{H-null} if $U_{H}=0$. 
Define the $H$-null \textbf{Socle}, $S_{H}(M)$  of $M$ as the largest (under inclusion) $H$-null $G$-graded $A$-submodule of $M$.
The tensor product $A\otimes_{A_{H}} N$ is the most natural $G$-graded $A$-module associated with a $H$-graded $A_{H}$-module $N$.
 We define \textit{induced module}  $A\overline{\otimes}_{A_{H}} N$ to be a $G$-graded $A$-quotient module:
$$ A\overline{\otimes}_{A_{H}}N=(A\otimes_{A_{H}}N)/S_{H}(A\otimes_{A_{H}} N).$$
\begin{definition}
A simple $G$-graded $A$-module $M$ is a $G$-graded \A -module with no proper $G$-graded submodule. $\textit{i.e.}$, the only $G$-graded $A$-submodules of $M$ are $0$ and itself.
\end{definition}
\begin{remark}\cite[Lemma 3.6 and Lemma 3.8]{Twisted}\label{simplegraded}
 Let $M$ be a simple $G$-graded  $A$-module then $M_{H}$ is either $0$ or a simple $H$-graded  $A_{H}$-module. 
 Also if $N$ is a simple $H$-graded $A_{H}$-module then $A\overline{\otimes}_{A_{H}} N$ is either $0$ or a simple $G$-graded $A$-module with $(A\overline{\otimes}_{A_{H}} N)_{H}\neq 0$. 
\end{remark}
\begin{remark}\label{partialaction}
Define the \textit{support} of a simple $G$-graded $A$-module $M$ as,
$$\Supp(M)=\{ g\in G: M_{g}\neq 0\}.$$
We define a \textit{partial action} of $G$  on simple $A_{1}$-modules as follows:\\
For $g\in G$ and for simple $A_{1}$-module $M$,
$$\prescript{g}{}{}M=(A\overline{\otimes}_{A_{1}}M)_{g}$$
by Remark \ref{simplegraded} $\prescript{g}{}{}M$ is a simple $A_{1}$-module, if $g \in \Supp(A\overline{\otimes}_{A_{1}} M)$ and $\prescript{g}{}{}M =0$, if $g \notin \Supp(A\overline{\otimes}_{A_{1}}M)$.
For more details see \cite[Section 7]{Dade1986}, \cite[Section 3.3]{Twisted}.
\end{remark}

\begin{definition}[$G$-graded Frobenius $\star$-algebra]
A $G$-graded Frobenius $\star$-algebra is a $G$-graded finite dimensional
associative unital $\C$-algebra $A=\bigoplus_{g\in G}A_g$ equipped with a semilinear anti-involution $\star:\A\rightarrow\A$ and
a  class functional $\lambda_{0}:\A_{1}\rightarrow\C$ (which is extended linearly to $\lambda:\A\to \C$ by zero on all non-trivial components), such that the following conditions hold:

$\bullet\ (A,\lambda,\s)$ is a Frobenius $\s$-algebra

$\bullet\ (\A_{g})^{\s}=\A_{g^{-1}},\forall\ g\in G$.
\end{definition}
\begin{remark}
Let $\A$ be a $G$-graded Frobenius $\star$-algebra then the trivial component  $\A_{1}$ of $\A$ is itself a Frobenius $\star$-algebra.
\end{remark}

\begin{definition}\label{def:centralizer}
Let $\A$ be a $G$-graded Frobenius $\star$-algebra.
The centralizer of $\A_{1}$ in $\A$ is denoted by $Z_{\A_{1}}(\A)$ and defined as 
$Z_{\A_{1}}(\A)=\{a\in \A:ab=ba\ \forall\ b\in\A_{1}\}$.
Moreover we have, 
$$Z_{\A_{1}}(\A)=\bigoplus_{g\in G}Z_{\A_{1}}(\A_{g})$$ and  $Z_{\A_{1}}(\A)$ is a $G$-graded Frobenius $\star$-algebra under the restrictions of $\star$ and $\lambda$. 
\end{definition}
\begin{remark}(cf.\cite[Lemma 4.3]{Twisted})\label{socle}
Let $A$ be a $G$-graded Frobenius $\star$-algebra and let $M$ be any  $A_{1}$-module.
Then the null socle of $G$-graded $A$-module $A\otimes_{A_{1}}M$ is zero.
Thus the $A$-module induced by the $A_{1}$-module $M$  is equal to $A\otimes_{A_{1}}M$.
\end{remark}

\begin{remark}\label{partial action}
Let $\A$ be a $G$-graded Frobenius $\s$-algebra and let $\Sim(\A_{1})$ denote the set of simple $\A_{1}$-modules upto isomorphism.
Then by Remark \ref{socle}, the partial action of $G$ on $\Sim(A_{1})$ is given as follows:

for $g\in G$ and $M\in\Sim(\A_{1})$,
$$\prescript{g}{}{}M=
\begin{cases}
(A\otimes_{A_{1}}M)_{g}\ \ \ \ \mbox{if}\ g\in\Supp(A\otimes_{A_{1}} M),\\
0\ \ \ \ \ \ \ \ \ \ \ \ \ \ \ \ \  \ \mbox{otherwise.}
\end{cases}$$
We will denote the set of fixed points for the partial action of $g\in G$ on $\Sim(A_{1})$ by $\Sim(A_{1})^{g}$ \emph{i.e.}
$\Sim(A_{1})^{g}=\{M\in\Sim(A_{1}): \prescript{g}{}{}M=M\}$.
\end{remark}
\begin{remark}\cite[Section 4]{Twisted}\label{extofmod}
Let $M$ be a simple $A_{1}$-module such that $\prescript{g}{}{}M\cong M$ as $A_{1}$-module.
Then we can make $M$ into a simple $A_{<g>}$-module in $o(g)$-ways and this extensions are unique upto the $o(g)^{th}$-root of unity. We will denote this $A_{<g>}$-module by $\widetilde{M}$.
\end{remark}
\begin{definition}[Twisted character]\label{def: twisted character}
(i) Let $A$ be a $G$-graded Frobenius $\star$-algebra and let $M\in\ \Sim(\A_{1})^{g}$ be any element.
Let $\chi_{M}$ denotes the character of $A_{1}$ corresponding to $M$. 
 By Remark \ref{extofmod},  $M$  becomes a simple  $A_{<g>}$-module and denote the corresponding character by $\chi_{\widetilde{M}}:\A_{<g>}\rightarrow\C$.
The \textit{$g$-twisted character} corresponding to $M$ is defined to be the linear functional $\chi_{M}^{g}:=\chi_{\widetilde{M}}|_{A_g}:A_g\to \C$. \\
(ii) By the definition of a $G$-graded Frobenius $\s$-algebra one can identify the dual of $A_{g}$ with $\A_{g^{-1}}$. 
For $M\in \Sim(\A_{1})^{g}$, let $\alpha_{M}^{g}\in\A_{g^{-1}}$ be the element corresponding to $\chi_{M}^{g}:A_g\to \C$. 
Sometimes we will call the element $\alpha_{M}^{g}$ as the \textit{$g$-twisted character} associated with $M\in\Rep^{g}$.
\end{definition}
\begin{remark}
Note that $\chi_{M}^{g}$, and hence $\alpha_{M}^{g}$ are well-defined up to scaling by $n^{th}$-roots of unity, where $n$ is the order of $g$ in $G$.
If $g=1$ then $\A_{1}$ is a Frobenius $\s$-algebra, so one can identify dual of $\A_{1}$   with $\A_{1}$.
Using the above identification we have the well- defined element $\alpha_{M}\in\A_{1}$ corresponding to $M\in\Sim(\A_{1})$.
\end{remark}

\begin{definition}[Twisted class functional]
 A linear functional $f$ on  $\A_{1}$-bimodule $E$ is said to be a twisted class functional if $f(am) = f(ma)\ \forall\ m\in E , a \in \A_{1}$. We will denote the set of all twisted class functionals on $E$ by $cf_{\A_{1}}(E)$ .
\end{definition}
\begin{remark}\cite[Lemma 4.12, Corollary 4.16]{Twisted}
Let $A$ be a $G$-graded Frobenius $\star$-algebra and 
 let $\Theta:A\rightarrow A^{\vee}$ be the corresponding $A$-module isomorphism.
 Then $\Theta^{-1}(cf_{A_{1}}(A_{g}))=Z_{A_{1}}(A_{g^{-1}})$
 and for $g\in G$, the set of $g$-twisted characters forms an orthogonal basis of $cf_{A_{1}}(A_{g})$.
\end{remark}
\begin{remark}
Let $\A$ be a $G$-graded Frobenius $\s$-algebra.
Let $\{e_{M}:M\in\Sim(A_{1})\}$ denote the set of primitive central orthogonal idempotents of $A_{1}$ and let
$E_{M}$ denote the simple $Z(A_{1})$-module $\mathbb{C}\cdot e_{M}$.
By the Remark  \ref{frob star:properties}, we have 
$$Z(\A_{1})=\bigoplus_{M\in\Sim(\A_{1})}E_{M}$$ 
and any simple $Z(\A_{1})$-module  is isomorphic to $E_{M}$ for some $M\in \Sim(\A_{1})$.
\end{remark}
\begin{remark}\label{decompo:centre of Frob star algebra}
\begin{enumerate}
\item By \cite[Lemma 4.15]{Twisted}, for any $g\in G$, there is a bijection between  the sets ${\Sim(\A_{1})}^{g}$ and ${\Sim(Z(\A_{1}))}^{g}$ (which is well-defined using Definition \ref{def:centralizer}).
\item  For every $M\in\Sim(\A_{1})$ and for any $g\in G$,  we have  $\prescript{g}{}{}(E_{M})=0 \ \mbox{or}\ \prescript{g}{}{}(E_{M})\cong E_{M}.$ 
Moreover, $$Z_{\A_{1}}(\A_{g^{-1}})\cong\bigoplus_{M\in\Sim(A_{1})^{g}}\C\cdot\alpha_{M}^{g}$$ as $Z(A_{1})$-module.
For more details see \cite[Lemma 3.3]{Oncenters}.
\end{enumerate}
\end{remark}

\begin{lemma}
Let $V$ be a finite dimensional $\C$-vecor space and   $\A$ be a $G$-graded Frobenius $\s$-algebra with $\A_{1}=\End(V)$.
Then for each $g\in G$,
$A_{g}=0$ or $\A_{g}\cong \End(V)$ as a $\A_{1}$-bimodule.
\end{lemma}
\begin{proof}
Note that for any ring $R$, $(R,R)$-bimodule is a $R\otimes R^{op}$-module.
For any  finite dimensional vector space $V$, 
$\End(V)\otimes\End(V)^{op}\cong\End(V\otimes V^{\vee})$.
Therefore, any $(\End(V),\End(V))$-bimodule is a semisimple module and  any simple $(\End(V),\End(V))$-bimodule  is isomorphic to $\End(V)$.
Then the result follows from the Remark \ref{decompo:centre of Frob star algebra}.

\end{proof}
\begin{lemma}
Let $\A$ be a $G$-graded Frobenius $\s$-algebra.
Then we have the direct sum decomposition of $\A_{g^{-1}}$,
\begin{equation}
    \A_{g^{-1}}=\bigoplus_{M\in\Sim(\A_{1})^{g}}M^{r_{m}}\ \mbox{as\ a}\ \A_{1}-\mbox{module} 
\end{equation}
for some $r_{m}\in\mathbb{N}$.
\end{lemma}
\begin{proof}
As $\A_{1}$ is a semisimple algebra, so $\A_{g^{-1}}$ is a semisimple $\A_{1}$-module \emph{i.e}
$$\A_{g^{-1}}=\bigoplus_{M\in\Sim(\A_{1})}M^{r_{m}}.$$
Then the result follows from the Remark \ref{decompo:centre of Frob star algebra}.
\end{proof}

\begin{remark}\label{remark: unitary and unit}
Let $R$ be a commutative Frobenius $\star$-algebra equipped with an action of $G$ by Frobenius $\star$-algebra automorphisms.
Let us define,  $U\coloneqq\{x\in R: x^{\star}=x^{-1}\}$ then we have a short exact sequence of $G$-modules
$$1\rightarrow U\rightarrow R^{\times}\rightarrow R^{\times}/U\rightarrow 1.$$
Note that $R^{\times}\cong (\C^{\times})^{\dim R}, U\cong (S^1)^{\dim R}$, hence the quotient
$R^{\times}/U$ is isomorphic to $(\mathbb{R}_{>0}^{\times})^{\dim R}$.
But $\mathbb{R}_{>0}^{\times}$ is an uniquely divisible group. Hence $H^{n}(G,R^{\times}/U)=0$ for all $n\geq 1$.
Using the long exact sequence of cohomology, we have
$$H^{k}(G,U)\cong H^{k}(G,R^{\times})\ \ \ \ \ \  \ \forall\ \  k\geq 2.$$
\end{remark}
\begin{prop}\label{Star:unique} 
Let $(A,\lambda,\star)$ be a $G$-graded Frobenius $\star$-algebra.
Let $\star_{1}:A\rightarrow A$ be an anti-involution such that $(A,\lambda,\star_{1})$ is also a $G$-graded Frobenius $\star$-algebra. 
Suppose $a^{\star_{1}}=a^{\star}$ for all $a\in A_{1}$. Then $x^{\star_{1}}=x^{\star}$ for all $x\in A$.
\end{prop}
\begin{proof}
Let $Z_{A_{1}}(A)=\oplus_{g\in G}Z_{A_{1}}(A_{g})$ be a $A_{1}$ center of $(A,\lambda,\star)$.
From the Remark \ref{decompo:centre of Frob star algebra}, we can choose the elements $e_{M}^{g}\in Z_{A_{1}}(A_{g})$, $e_{M}^{g^{-1}}\in Z_{A_{1}}(A_{g^{-1}})$ corresponding to each $M\in\Sim(A_{1})^{g}$ such that
 $$Z_{A_{1}}(A_{g})=\bigoplus_{M\in \Sim(A_{1})^{g}} \C\cdot e_{M}^{g}\ \ \mbox{and}\ \ Z_{A_{1}}(A_{g^{-1}})=\bigoplus_{M\in \Sim(A_{1})^{g}} \C\cdot e_{M}^{g^{-1}}.$$
 Also one can choose these elements such that $(e_{M}^{g})^{\star}=e_{M}^{g^{-1}}$.
 By the definition of a $G$-graded Frobenius $\star$-algebra,
 $$(e_{M}^{g})^{\star_{1}}=f_{M}(g)e_{M}^{g^{-1}}\ \ \ \ \ \mbox{for some} \ f_{M}(g)\in \C^{\times}.$$
 We have  $0<\lambda(e_{M}^{g}(e_{M}^{g})^{\star_{1}})=\lambda(e_{M}^{g}f_{M}(g)e_{M}^{g^{-1}})$ this implies that $f_{M}(g)>0$.
 Suppose $G_{M}$ denotes the stabilizer of $M\in\Sim(A_{1})$ in $G$ under the partial action.
 Then the anti-involution implies that $f_{M}:G_{M}\rightarrow \C^{\times}$ is a group homomorphism. 
 This implies that $f_{M}(g)=1$ for all $g\in G_{M}$ and for all $M\in\Sim(A_{1})$.
 Therefore $(e_{M}^{g})^{\star_{1}}=e_{M}^{g^{-1}}=(e_{M}^{g})^{\star}$ for all $g\in G_{M}$ and for all $M\in\Sim(A_{1})$.
 
 Let $x\in A_{g}$ be any element then $x=\sum_{M\in \Sim(A_{1})^{g}}a_{m}e_{M}^{g}$ for some $a_{m}\in A_{1}$.
 So
 $$x^{\star_{1}}=\sum_{M\in \Sim(A_{1})^{g}}(e_{M}^{g})^{\star_{1}}a_{m}^{\star_{1}}=\sum_{M\in \Sim(A_{1})^{g}}(e_{M}^{g})^{\star}a_{m}^{\star}=x^{\star}\ \ \ \ \ \ (\mbox{since}\ a^{\star_{1}}=a^{\star} \forall\ a\in A_{1}).$$
 This proves that $x^{\star_{1}}=x^{\star}$ for all $x\in A$.
 \end{proof}

\section{ $G$-crossed Frobenius $\star$-algebras}
\label{sec:crossed algebra}
Now we will define $G$-crossed Frobenius $\s$-algebras and prove some results about them.
\begin{definition}[$G$-crossed Frobenius $\star$-algebra]\label{def: crossed algebra}
A $G$-crossed Frobenius $\s$-algebra $\A$ is a $G$-graded Frobenius $\star$-algebra equipped with an action of $G$ on $A$ by Frobenius $\star$-algebra automorphism such that following  holds:
\begin{itemize}
\item For each $g\in G,\ g:A\rightarrow A$ is a Frobenius $\star$-algebra automorphism of $\A$.
\item $g(\A_{h})\subseteq\A_{ghg^{-1}}\  \ \forall\ g,h\in G.$
\item $ ab=g(b)a\ \ \forall\ a\in\A_{g},\forall\ b\in \A$ (Crossed commutativity condition).
\end{itemize}
\end{definition}
\begin{remark}\label{crossed algebra:properties}
Let $\A$ be a $G$-crossed Frobenius $\s$-algebra then the following holds:
\begin{enumerate}
\item $\A_{1}$ is commutative Frobenius $\s$-algebra.
\item For each $g\in G$, $g(A_{1})\subseteq A_{1}$.
\item For any $g\in G$, $Z_{\A_{1}}(\A_{g})=\A_{g}$.
\end{enumerate}
\end{remark}

\begin{example}
$1.$ The complexified Grothendieck algebra $K(\D)$ of a $G$-crossed braided fusion category $\D$ is a   $G$-crossed Frobenius $\star$-algebra.

$2.$ Let $A$ be a strongly $G$-graded extension of a commutative Frobenius $\star$-algebra $R$ then $Z_{R}(A)$
is a $G$-crossed Frobenius $\star$-algebra. 
(For more details see \textsection \ref{sec:strongly graded algebra}.)
\end{example}
\subsection{Strict $G$-crossed Frobenius $\star$-algebras}\label{sec:strict}
Let $\A$ be a $G$-crossed Frobenius $\s$-algebra.
By Schur's lemma, every irreducible representation of $\A_{1}$ is $1$-dimensional.
Suppose $\Sim(\A_{1})$ denotes the  set of irreducible representations of $\A_{1}$ \emph{i.e.}
$$\Sim(\A_{1})=\{\chi:\A_{1}\rightarrow\C:\chi\ \mbox{is\ algebra\ homomorphism}\}.$$
From the Remark \ref{crossed algebra:properties}, there is a natural action of $G$ on $\Sim(\A_{1})$ as follows:
for $\chi\in\Sim(\A_{1})$ and $g\in G$,
\begin{equation}\label{action}
\prescript{g}{}{}\chi=\chi\circ g^{-1}.
\end{equation}
\begin{lemma}\label{lem:fixedpoints}
Let $A$ be a $G$-crossed Frobenius $\star$-algebra.
 Suppose  $\chi\in\Sim(\A_{1})$ is fixed by  the partial action of some $g\in G$ on $\Sim(\A_{1})$ coming from a $G$-grading (as in Remark \ref{partial action}). Then it is also fixed by the  action of $g$ defined above (\ref{action}).
 In other words, the set of fixed points for the partial action of $g$ on $\Sim(A_{1})$   is contained in the set of fixed points for the action of $g$ on $\Sim(A_{1})$ defined above (\ref{action}).
\end{lemma}
\begin{proof}
Let $e_{\chi}$ denote the primitive central orthogonal idempotent of $A_{1}$ corresponding to character $\chi\in\Sim(A_{1})$.
By definition of $G$-crossed Frobenius $\star$-algebra,
we have $e_{\chi}\cdot A_{g}=g(e_{\chi})\cdot A_{g}$.
If $\chi\circ g^{-1}\neq \chi$ then by Remark \ref{decompo:centre of Frob star algebra}, we must have $e_{\chi}\cdot A_{g}=0$ and hence,  $\chi$ is not fixed by partial action.
This proves the result.
\end{proof}
\begin{definition}[Strict $G$-crossed Frobenius $\star$-algebra]
 Let $A$ be a $G$-crossed Frobenius $\star$-algebra. We say that $A$ is a \emph{strict} $G$-crossed  Frobenius $\star$-algebra if
for each $g\in G$ the set of fixed points for the partial action of $g$ on $\Sim(\A_{1})$ is equal to the set of fixed points for the  action of $g$ on $\Sim(\A_{1})$ from (\ref{action}), i.e. if the inclusion from Lemma \ref{lem:fixedpoints} is an equality for each $g\in G$. In this case the notation $\Sim(A_1)^g$ is unambiguous.
\end{definition}
\begin{lemma}\label{properties: G crossed algebra}
Let $\A$ be a strict $G$-crossed Frobenius $\s$-algebra. 
\begin{enumerate}
\item[(i)] Let $\chi,\chi'\in\Sim(\A_{1})$ then $\chi(\alpha_{\chi'})=0$ if $\chi\neq\chi'$ and the element $e_{\chi}\coloneqq\frac{\alpha_{\chi}}{\chi(\alpha_{\chi})}$ is the  primitive central  orthogonal idempotent corresponding to $\chi\in\Sim(\A_{1})$.
\item[(ii)] For $\chi,\chi'\in\Sim(\A_{1})$, we have the orthogonality relations:
$$\langle\alpha_{\chi},\alpha_{\chi'}\rangle=
\begin{cases}
0 &\mbox{if}\ \chi\neq\chi';
\\
\overline{\chi(\alpha_{\chi})}\lambda(\alpha_{\chi})&\mbox{if}\ \chi=\chi'.
\end{cases}$$
\item[(iii)] Let $\chi\in\Sim(\A_{1})^{g}$ and $\alpha_{\chi}^{g}$ be the corresponding $g$-twisted character (\ref{def: twisted character}).
Then for any $\chi'\in\Sim(\A_{1})$, $$\alpha_{\chi'}\cdot\alpha_{\chi}^{g}=\alpha_{\chi}^{g}\cdot\alpha_{\chi'}=\begin{cases}
0 &\mbox{if}\ \chi\neq\chi';
\\
\chi(\alpha_{\chi})\alpha_{\chi}^{g}&\mbox{if}\ \chi=\chi'.
\end{cases}$$
\item[(iv)] We have the direct sum decomposition of $\A_{g^{-1}}$,
\begin{equation}
    \A_{g^{-1}}\cong\bigoplus_{\chi\in\Sim(\A_{1})^{g}}\C\cdot\alpha_{\chi}^{g}\ \ \mbox{as\ a}\ \A_{1}-\mbox{module}.
\end{equation}
\end{enumerate}
\end{lemma}
\begin{proof}
Parts $(i),\ (ii)$ follows from the fact that $\A_{1}$ is a commutative semisimple algebra and $(iv)$ follows from the Remark \ref{decompo:centre of Frob star algebra}.

For part $(iii)$, we have a $\A$-module  isomorphism $\Theta:\A\rightarrow\A^{\vee}$.
By the definition of twisted character $\alpha_{\chi}^{g}=\Theta^{-1}(\chi^{g})$.
Suppose  $\chi\neq\chi'\in\Sim(\A_{1})$ then 
$e_{\chi'}\cdot\chi^{g}=0$.
Consider $0=\Theta^{-1}(e_{\chi'}\cdot\chi^{g})=e_{\chi'}\cdot\Theta^{-1}(\chi^{g})=e_{\chi'}\cdot\alpha_{\chi}^{g}$.
Therefore $\alpha_{\chi'}\cdot\alpha_{\chi}^{g}=\chi(\alpha_{\chi'})e_{\chi'}\cdot\alpha_{\chi}^{g}=0.$
Also $\alpha_{\chi}\cdot\alpha_{\chi}^{g}=\chi(\alpha_{\chi})e_{\chi}\cdot\alpha_{\chi}^{g}=\chi(\alpha_{\chi})\alpha_{\chi}^{g}.$
This proves the result.
 \end{proof}
 
\begin{corollary}\label{coro: properties}
Let $\A$ be a strict $G$-crossed Frobenius $\star$-algebra then
\begin{enumerate}
\item[(i)] Let  $\chi\in\Sim(\A_{1})^{g}$, $\chi'\in\Sim(\A_{1})^{h}$ be two twisted characters.
Then for $\chi\neq\chi'$,
$\alpha_{\chi}^{g}\cdot\alpha_{\chi'}^{h}=0.$
\item[(ii)] $\dim_{\C}(\A_{g^{-1}})=\dim_{\C}(\A_{g})=|\Sim(\A_{1})^{g}|$.
\item[(iii)] $\{\alpha_{\chi}^{g}:\chi\in\Sim(\A_{1})^{g}\}$ forms an orthogonal basis of $\A_{g^{-1}}$.
\item[(iv)] Let $\chi\in\Rep^{g}$ and $\alpha_{\chi}^{g}$ be the corresponding $g$-twisted character.
Define $e_{\chi}^{g}\coloneqq\frac{\alpha_{\chi}^{g}}{\chi(\alpha_{\chi})}$ then $(e_{\chi}^{g})^{n}=e_{\chi}$, where $n$ is the order of $g$ in $G$.
\end{enumerate}
\end{corollary}
\begin{proof}
$(i), (ii)$ and $(iii)$ directly follows from the Lemma \ref{properties: G crossed algebra}. \\
For $(iv)$, let $\chi\in\Sim(A_{1})^{g}$ then by \cite[(1)]{Twisted} we have $\langle\alpha_{\chi}^{g},\alpha_{\chi}^{g}\rangle=\langle\alpha_{\chi},\alpha_{\chi}\rangle$.
By using the definition of the Hermitian form and twisted characters,  we have $$\langle\alpha_{\chi}^{g},\alpha_{\chi}^{g}\rangle=\lambda(\alpha_{\chi}^{g}{\alpha_{\chi}^{g}}^{\star})=
\lambda({\alpha_{\chi}^{g}}^{\star}\alpha_{\chi}^{g})=
\widetilde{\chi}({\alpha_{\chi}^{g}}^{\star})\ \ \  
\mbox{and}\ \ \langle\alpha_{\chi},\alpha_{\chi}\rangle=\chi(\alpha_{\chi}^{\star})=\overline{\chi(\alpha_{\chi})}$$ where $\widetilde{\chi}$ as in Remark \ref{extofmod}.
Hence $\widetilde{\chi}({\alpha_{\chi}^{g}}^{\star})=\overline{\chi(\alpha_{\chi})}$.
Also we have  $\alpha_{\chi}^{g}{\alpha_{\chi}^{g}}^{\star}=\chi(\alpha_{\chi}^{g}{\alpha_{\chi}^{g}}^{\star})e_{\chi}$ implies that
$$\langle\alpha_{\chi}^{g},\alpha_{\chi}^{g}\rangle=\chi(\alpha_{\chi}^{g}{\alpha_{\chi}^{g}}^{\star})\lambda(e_{\chi})=\frac{\widetilde{\chi}(\alpha_{\chi}^{g})\widetilde{\chi}({\alpha_{\chi}^{g}}^{\star})\lambda(\alpha_{\chi})}{\chi(\alpha_{\chi})}$$
and by Lemma \ref{properties: G crossed algebra}, $\langle\alpha_{\chi},\alpha_{\chi}\rangle=\overline{\chi(\alpha_{\chi})}\lambda(\alpha_{\chi})$.
Hence, $\widetilde{\chi}(\alpha_{\chi}^{g})=\chi(\alpha_{\chi})$.
Consider, $$\chi((e_{\chi}^{g})^{n})=\chi(\frac{(\alpha_{\chi}^{g})^{n}}{\chi(\alpha_{\chi})^{n}})=\frac{\chi(\alpha_{\chi}^{g})^{n})}{\chi(\alpha_{\chi})^{n}}=\frac{\widetilde{\chi}(\alpha_{\chi}^{g})^{n}}{\chi(\alpha_{\chi})^{n}}=1.$$
By $(iii)$ of Lemma \ref{properties: G crossed algebra}, $(e_{\chi}^{g})^{n}=e_{\chi}$.
This proves the result.
\end{proof}

For $\chi\in\Sim(\A_{1})$,
let $G_{\chi}$ be  the stabilizer of $\chi$ under the action of $G$ \emph{i.e.} $G_{\chi}=\{g\in G:\prescript{g}{}{}\chi=\chi\}$.
Using the previous results we get
 a  $G_{\chi}$-graded algebra $\alpha_{\chi}\A$, $$\alpha_{\chi}\A=\bigoplus_{g\in G}\alpha_{\chi}\cdot\A_{g}
    \cong\bigoplus_{g\in G_{\chi}}\C\cdot e_{\chi}^{g}.$$ 
From this $G_{\chi}$-graded algebra $\alpha_{\chi}\A$, we obtain a central extension $G'_{\chi}$ of $G_{\chi}$ by $\C^{\times}$ \emph{i.e} 
$$1\rightarrow\C^{\times}\rightarrow G'_{\chi}\rightarrow G_{\chi}\rightarrow 1.$$
Note that for $g,h\in G_{\chi}$ the product $e_{\chi}^{g}\cdot e_{\chi}^{h}\in\A_{gh}$.
Our choice of elements $e_{\chi}^{g}$ determines a $2$-cocycle $\phi_{\chi}:G_{\chi}\times G_{\chi}\rightarrow\C^{\times}$ such that
$e_{\chi}^{g}\cdot e_{\chi}^{h}=\phi_{\chi}(g,h)e_{\chi}^{gh}$.
Moreover one can choose these  elements such that   $e_{\chi}^{g}e_{\chi}^{g^{-1}}=e_{\chi}$.
Hence we obtain:

\begin{corollary}\label{cocycle}
Let $\chi\in\Rep$ and $g,h\in G_{\chi}$ be any element 
then $e_{\chi}^{g}\cdot e_{\chi}^{h}=\phi_{\chi}(g,h)e_{\chi}^{gh}$, where $\phi_{\chi}$ is the $2$-cocycle corresponding to an extension $$1\rightarrow\C^{\times}\rightarrow G'_{\chi}\rightarrow G_{\chi}\rightarrow 1.$$
\end{corollary}
\begin{remark}\label{scalar}
Let $\chi\in\Sim(A_{1})$.
For any $g_{1},g_{2},\cdots,g_{n}\in G_{\chi}$, let $\phi_{\chi}(g_{1},\cdots,g_{n})$ be such that\\ 
$\chi^{g_{1}}(a_{1})\chi^{g_{2}}(a_{2})\cdots\chi^{g_{n}}(a_{n})=\phi_{\chi}(g_{1},\cdots,g_{n})\chi^{g_{1}g_{1}\cdots g_{n}}(a_{1}a_{2}\cdots a_{n})$.
These scalars will appear in our twisted Verlinde formula. 
\end{remark}

\begin{prop}\label{G-crosed str}
Let $\A=\oplus_{g\in G}A_{g}$ be a strict $G$-crossed Frobenius $\star$-algebra then following holds:
\begin{enumerate}
\item For each $g\in G_{\chi}$, we have a $g$-twisted character $\alpha_{\chi}^{g}$ and element  $e_{\chi}^{g}$ (determined up to scaling by root of unity) such that 
$$\A_{g^{-1}}\cong\bigoplus_{\chi\in \Sim(\A_{1})^{g}}\C\cdot e_{\chi}^{g}\ \ \mbox{and}\ \ e_{\chi}e_{\chi'}^{g}=\delta_{\chi,\chi'}e_{\chi}^{g}.$$
\item Let $\chi\in\Sim(\A_{1})$ be any element 
then $$e_{\chi}A\cong\bigoplus_{g\in G_{\chi}}\C\cdot e_{\chi}^{g}$$ is a $G_{\chi}$-crossed Frobenius $\star$-algebra.
There exists a normalized $2$-cocycle  $\phi_{\chi}\in H^{2}(G_{\chi},\C^{\times})$
such that the multiplication and the action  $G_{\chi}$ on $e_{\chi}A$  is determined as: 
for all $g,h\in G_{\chi}$,
$$e_{\chi}^{g}e_{\chi}^{h}=\phi_{\chi}(g,h)e_{\chi}^{gh}\ \ \mbox{and}\ \  g(e_{\chi}^{h})=\frac{\phi_{\chi}(g,h)}{\phi_{\chi}(ghg^{-1},g)}e_{\chi}^{ghg^{-1}}.$$
\item There exists a $2$-cocycle $\phi\in H^{2}(G,A_{1}^{\times})$ such that  for $\chi\in\Sim(A_{1})$, $$\chi\circ\phi(g,h)=\phi_{\chi}(g,h)\ \ \mbox{and} \ \ e_{\chi}^{g}e_{\chi}^{h}=\phi(g,h)e_{\chi}^{gh}\ \ \ \forall\  g,h\in G_{\chi}.$$ 
Moreover, the given action of $G$ on $\A$ is determined as, for each $g,h\in G$, $\chi\in\Sim(\A_{1})^{h}$
$$g(e_{\chi}^{h})=\frac{\phi(g,h)}{\phi(gh g^{-1},g)} e_{\prescript{g}{}{\chi}}^{ghg^{-1}}.$$
\item Let $\psi:G\times G\rightarrow \A_{1}^{\times}$ defined as $\psi(g,h)=\frac{\phi(g,h)^{\star}}{\phi(h^{-1},g^{-1})}$, where $\phi$ as in $(3)$ above.
Then there exist a $1$-cochain $\theta:G\rightarrow \A_{1}$ such that $\partial(\theta)=\psi$ and the $\star$-map is given as, for each $e_{\chi}^{g}\in \A_{g}$,  $(e_{\chi}^{g})^{\star}=\theta(g)e_{\chi}^{g^{-1}}\ \ \forall\ g\in G,\ \forall\ \chi\in \Sim(\A_{1})^{g}$.
\end{enumerate}
\end{prop}
\begin{proof}
 $(1)$ follows from the Lemma \ref{properties: G crossed algebra} and Corollary \ref{coro: properties}. 

For $(2)$, let $\chi\in\Sim(\A_{1})$ be any element then from Corollary \ref{cocycle}, there exists a $2$-cocycle
$\phi_{\chi}$ in $H^{2}(G_{\chi},\C^{\times})$ such that 
$e_{\chi}^{g}\cdot e_{\chi}^{h}=\phi_{\chi}(g,h)e_{\chi}^{gh}\ \forall\ g,h\in G_{\chi}$.
By the definition of a $G$-crossed Frobenius $\star$-algebra, we have
$ e_{\chi}^{g}e_{\chi}^{h}=g( e_{\chi}^{h})e_{\chi}^{g}$ and $g(e_{\chi}^{h})\in \A_{ghg^{-1}}$, so $g(e_{\chi}^{h}) =\beta\cdot e_{\chi}^{ghg^{-1}}$ for some $\beta\in \C^{\times}$.

Now consider
$$\phi_{\chi}(g,h)e_{\chi}^{gh}=e_{\chi}^{g}\cdot e_{\chi}^{h}=g( e_{\chi}^{h})e_{\chi}^{g}=\beta\cdot e_{\chi}^{ghg^{-1}}e_{\chi}^{g}=\beta\phi_{\chi}(ghg^{-1},g)e_{\chi}^{gh}.$$
This implies  $$g(e_{\chi}^{h})=\frac{\phi_{\chi}(g,h)}{\phi_{\chi}(ghg^{-1},g)}e_{\chi}^{ghg^{-1}}\ \ \ \forall\ g,h\in G_{\chi}.$$

For $(3)$,
one can easily check that $\A_{1}^{\times}=\oplus\CoInd_{G_{\chi}}^{G}(\C^{\times})$,
where the sum runs over the orbit representatives for the induced action of $G$ on $\Sim(\A_{1})$.
Then by Shapiro's Lemma, we have $$H^{2}(G,\CoInd_{G_{\chi}}^{G}(\C^{\times}))\cong H^{2}(G_{\chi},\C^{\times}).$$
Also $H^{2}(G,\A_{1}^{\times})\cong\oplus H^{2}(G,\CoInd_{G_{\chi}}^{G}(\C^{\times}))$, where the sum runs over the orbit representatives for the induced action of $G$ on $\Sim(\A_{1})$.
Hence we get a required $\phi\in H^{2}(G,\A_{1}^{\times})$.

$(4)$ follows from the part $(2)$ and the  fact that
for any $x,y\in \A, \ (xy)^{\star}=y^{\star}x^{\star}$ and $(e_{\chi}^{g})^{\star}$ is a  scalar multiple of $e_{\chi}^{g^{-1}}$.
This proves the result.

\end{proof}

\subsection{Proof of the main result}
Now we will define and classify the strict $G$-crossed extension of a commutative Frobenius $\star$-algebra.
\begin{definition}[$G$-crossed extension]
 Let $G$ be a finite group and $R$ be a commutative Frobenius $\star$-algebra such that $G$ acts on $R$ by  Frobenius $\star$-algebra automorphism. 
 Then \emph{$G$-crossed extension}  of $R$ is a pair $(\A,j)$, where $A=\oplus_{g\in G}\A_{g}$ is a $G$-crossed Frobenius $\star$-algebra, $j:R\hookrightarrow\A$ is a Frobenius $\star$-algebra embedding with $j(R)=\A_{1}$ and restriction of the action of $G$ on $j(R)$    agrees with the given action of $G$ on $R$.
 
 If $A$ is strict $G$-crossed Frobenius $\star$-algebra then we will called   $(A,j)$   as a strict $G$-crossed extension of $R$.
 We will denote the set of all isomorphism classes of strict $G$-crossed extensions of  $R$  by $\Frob(G,R)$.
\end{definition}

\noindent Now we will complete the proof of the main result: \begin{proof}
[\bf{Proof of Theorem \ref{main result}}]
Let $A$ be a strict $G$-crossed extension of $R$.
By the Prop. \ref{G-crosed str}, $\A=\oplus_{g\in G}\A_{g}$ as $R$-module and for each $g\in G$ there exists  $\{e_{\chi}^{g}:\chi\in \Sim(A_{1})^{g}\}$ such that $\A_{g^{-1}}\cong\oplus_{\chi\in\Sim(R)^{g}} \C\cdot e_{\chi}^{g}$ .
Also there exists a  normalized $2$-cocycle $\phi\in C^{2}(G,R^{\times})$ such that  multiplication and representation of $G$ on $\A$ is as follows:
$$e_{\chi}^{g}e_{\chi}^{h}=\phi(g,h)e_{\chi}^{gh},\ \  g(e_{\chi}^{h})=\frac{\phi(g,h)}{\phi(ghg^{-1},g)}e_{\prescript{g}{}{\chi}}^{ghg^{-1}}.$$
We will denote $A$ with the above operations by $\A^{\phi}$.
Let $A^{\phi'}$ be an another strict $G$-crossed extension of $R$ isomorphic to $A^{\phi}$ then there exists a grade preserving isomorphism $f:A^{\phi'}\rightarrow A^{\phi}$ which is identity on $1$-grade.
Using this isomorphism we get a function $\theta:G\rightarrow R^{\times}$ such that $f({e'}_{\chi}^{g})=\theta(g)e_{\chi}^{g}$, for all $\chi\in\Sim(R)^{g}$.
Also $f$ is a homomorphism implies that
 $\phi'(g,h)=\partial(\theta)(g,h)\phi(g,h)$ for all $g,h\in G$ \emph{i.e.} $[\phi']=[\phi]\in H^{2}(G,R^{\times})$.
Thus, we get a well-defined map 
$$F: \Frob(G,R)\rightarrow H^{2}(G,R^{\times})\ \ \mbox{defined as}\ \ F(A^{\phi})= [\phi].$$

Let $\phi\in H^{2}(G,R^{\times})$ be a normalized $2$-cocycle. 
Let $A_{g^{-1}}$ be a $\C$-vector space with basis $\{e_{\chi}^{g}:\chi\in \Sim(R)^{g}\}$.
We define a $R$-bimodule structure on $A_{g^{-1}}$ as, for $\chi\in \Sim(R),\chi'\in \Sim(R)^{g}$ $$e_{\chi}e_{\chi'}^{g}=e_{\chi'}^{g}e_{\chi}=\delta_{\chi,\chi'}e_{\chi'}^{g}$$
extend $\C$-linearly to $A_{g^{-1}}$.

Suppose   $A^{\phi}=\oplus_{g\in G}A_{g}$ then $A^{\phi}$ is a $G$-graded $R$-bimodule.
Now we define a multiplication, action of $G$  on basis elements using $\phi$ as, $$e_{\chi}^{g}e_{\chi}^{h}=\phi(g,h)e_{\chi}^{gh},\ \mbox{and}\ \  g(e_{\chi}^{h})=\frac{\phi(g,h)}{\phi(ghg^{-1},g)}e_{\prescript{g}{}{\chi}}^{ghg^{-1}}$$
and extend it $R$-linearly to $A^{\phi}$.
Define $(e_{\chi}^{g})^{\star}=e_{\chi}^{g^{-1}}$  extend it $R$-conjugate linearly to $A^{\phi}$ then in  view of Remark \ref{remark: unitary and unit}, $\star$-map is an anti-involution.
Then $A^{\phi}$ becomes a $G$-crossed Frobenius $\star$-algebra.
We will show that $[\A^{\phi}]$ in $\Frob(G,R)$ is depends only on $[\phi]$ in $H^{2}(G,R^{\times})$.
Let $\phi'$ be another $2$-cocycle such that $[\phi]=[\phi']$ in $H^{2}(G,R^{\times})$ then there exists $\theta:G\rightarrow R^{\times}$ such that $\phi=\partial(\theta)\phi'$.
Define $f:\A^{\phi}\rightarrow \A^{\phi'}$ as $f(e_{\chi}^{g})=\theta(g){e}_{\chi}^{g}$ for $\chi\in\Sim(R),\ g\in G$ and extend $R$-linearly to $\A$ then one can easily check that $f$ is an isomorphism of $G$-crossed extensions.
This proves that $[\A^{\phi}]$ in $\Frob(G,R)$ is depends only on $[\phi]\in H^{2}(G,R^{\times})$.
Thus, we get a well-defined map
 $$F_{1}:H^{2}(G,R^{\times})\rightarrow \Frob(G,R)\ \ \mbox{defined as}\ \ F_{1}([\phi])= \A^{\phi}.$$
One can check that $F$ and $F_{1}$ are inverses of each other.
This proves the theorem.
\end{proof}

\subsection{Twisted Verlinde Formula}

 We will now  state and prove a twisted Verlinde  formula for  the computation of the linear functional $\lambda$ on a strict  $G$-crossed Frobenius $\star$-algebra.
 The twisted Verlinde formula in case of braided $G$-crossed categories was proved in \cite{Deshpande2019CrossedMC}.
\begin{lemma}\label{stabilizer}
Let $A$ be a strict $G$-crossed Frobenius $\star$-algebra.
Let $g_{1},\cdots,g_{n}$ be any $n$-element of $G$ and let $a_{i}\in A_{g_{i}}$, so that $a_{1}a_{2}\cdots a_{n}\in A_{g_{1}g_{2}\cdots g_{n}}$.
Let $\chi\in\Sim(A_{1})^{g_{1}g_{2}\cdots g_{n}}$ such that ($g_{1}g_{2}\cdots g_{n}$-twisted character) 
$\chi^{g_{1}g_{2}\cdots g_{n}}(a_{1}a_{2}\cdots a_{n})\neq 0$. 
Then $\chi\in \Sim(A_{1})^{\langle g_{1},g_{2},\cdots, g_{n}\rangle}$.
\end{lemma}
\begin{proof} 
Consider any $a\in A_{1}$.
Then by the definition of $G$-crossed Frobenius $\star$-algebra, for each $i$ we have 
$$a\cdot a_{1}a_{2}\cdots a_{i}=a_{1}a_{2}\cdots a_{i}\cdot a=g_{1}g_{2}\cdots g_{i}(a)\cdot a_{1}a_{2}\cdots a_{i}.$$
Hence for each $i$, $a(a_{1}a_{2}\cdots a_{n})=(g_{1}g_{2}\cdots g_{i})(a)(a_{1}a_{2}\cdots a_{n})$ as an element of $A_{g_{1}g_{2}\cdots g_{n}}$.
This implies that, $$\chi^{g_{1}g_{2}\cdots g_{n}}(aa_{1}a_{2}\cdots a_{n})=\chi^{g_{1}g_{2}\cdots g_{n}}((g_{1}g_{2}\cdots g_{i})(a)a_{1}a_{2}\cdots a_{n}).$$
Hence, $\chi(a)\chi^{g_{1}g_{2}\cdots g_{n}}(a_{1}a_{2}\cdots a_{n})=\chi\circ(g_{1}g_{2}\cdots g_{i})(a)\chi^{g_{1}g_{2}\cdots g_{n}}(a_{1}a_{2}\cdots a_{n})$.
Since we have assumed that $\chi^{g_{1}g_{2}\cdots g_{n}}(a_{1}a_{2}\cdots a_{n})\neq 0$.
Therefore, $$\chi(a)=\chi\circ(g_{1}g_{2}\cdots g_{i})(a) \ \ \ \  \mbox{for any}\ a\in A_{1}\ \mbox{and}\ 1\leq i\leq n.$$
This implies that $\chi$ is fixed by each $g_{i}$.
This proves the result.
\end{proof}
\begin{theorem}[twisted Verlinde formula in genus 0]\label{Verlinde formula:genus 0}
Let $A$ be a strict $G$-crossed Frobenius $\star$-algebra.
Let $g_{1},g_{2},\cdots,g_{n}\in G$ such that $g_{1}g_{2}\cdots g_{n}=1$. 
Let $a_{i}\in A_{g_{i}}$ then 
$$\lambda(a_{1}a_{2}\cdots a_{n})=\sum_{\chi\in \Sim(A_{1})^{\langle g_{1},g_{2},\cdots, g_{n}\rangle}}\frac{\chi^{g_{1}}(a_{1})\chi^{g_{2}}(a_{2})\cdots\chi^{g_{n}}(a_{n})}{\chi(\alpha_{\chi})\cdot\chi\circ\phi(g_{1},g_{2},\cdots,g_{n})}$$
where  $\chi\circ\phi(g_{1},g_{2},\cdots,g_{n})=\phi_{\chi}(g_{1},g_{2},\cdots,g_{n})$ are the scalar defined in Remark \ref{scalar}. 
\end{theorem}
\begin{proof}
As $A_{1}$ is a semisimple algebra.
Therefore,
$$1=\sum_{\chi\in\Sim(A_{1})}e_{\chi}=\sum_{\chi\in\Sim(A_{1})}\frac{\alpha_{\chi}}{\chi(\alpha_{\chi})}.$$
Then by definition of Frobenius $\star$-algebra, we have
$$\lambda=\sum_{\chi\in\Sim(A_{1})}\frac{\chi}{\chi(\alpha_{\chi})}.$$
Thus \begin{equation*} \begin{split}
\lambda(a_{1}a_{2}\cdots a_{n}) &=\sum_{\chi\in\Sim(A_{1})}\frac{\chi(a_{1}a_{2}\cdots a_{n})}{\chi(\alpha_{\chi})}\\
&= \sum_{\chi\in\Sim(A_{1})^{\langle g_{1},g_{2},\cdots,g_{n}\rangle}}\frac{\chi(a_{1}a_{2}\cdots a_{n})}{\chi(\alpha_{\chi})}\ \ (\mbox{by\ Lemma}\ \ref{stabilizer}) \\
&=\sum_{\chi\in \Sim(A_{1})^{\langle g_{1},g_{2},\cdots, g_{n}\rangle}}\frac{\chi^{g_{1}}(a_{1})\chi^{g_{2}}(a_{2})\cdots\chi^{g_{n}}(a_{n})}{\chi(\alpha_{\chi})\cdot\chi\circ\phi(g_{1},g_{2},\cdots,g_{n})}.
\end{split}
\end{equation*}
The last equality follows from the Remark \ref{scalar}.
\end{proof}
Let $A$ be a strict $G$-crossed Frobenius $\star$-algebra.
Let $g,h\in G$ be any elements.
Suppose $P_{g}$ denotes an orthonormal basis of $A_{g}$.
We define the element 
$$\Omega_{g,h}\coloneqq\sum_{x\in P_{g}}x\cdot h(x^{\star})\in A_{ghg^{-1}h^{-1}}.$$
Let $P'_{g}$ be an another orthonormal basis of $A_{g}$ and let $U$ be the  change of basis matrix. Then $U$ is a $P'_{g}\times P_{g}$-unitary matrix.
Consider,
\begin{equation*}
\sum_{y\in P'_{g}}y\cdot h(y^{\star})=\sum_{x,x'\in P_{g}}\left( \sum_{y\in P'_{g}}U_{y,x}\overline{U_{y,x'}}\right)x\cdot h(x'^{\star})=\sum_{x\in P_{g}}x\cdot h(x^{\star}).
\end{equation*}
The last equality follows from the unitarity of $U$.
This implies  that the element $\Omega_{g,h}$ is independent of choice of orthonormal basis.
\begin{lemma}\label{lemma:verlinde}
Let $g,h\in G$ and let $\chi\in\Sim(A_{1})^{[g,h]}$ be a character fixed by the commutator $[g,h]=ghg^{-1}h^{-1}$.
Then $$\chi^{[g,h]}(\Omega_{g,h})=
\begin{cases}
\frac{\chi(\alpha_{\chi})}{\chi\circ\phi(g,h,g^{-1},h^{-1})}\ \ \ \ \ \mbox{if}\ \chi\in\Sim(A_{1})^{<g,h>}\\
0\ \ \ \ \ \ \ \ \ \ \ \ \ \ \ \ \ \ \ \ \mbox{otherwise.}
\end{cases}$$
\end{lemma}
\begin{proof}
Consider,
$$\chi^{[g,h]}(\Omega_{g,h})=\chi^{[g,h]}\left(\sum_{x\in P_{g}}x\cdot h(x^{\star})\right)=\sum_{x\in P_{g}}\chi^{[g,h]}(x\cdot h(x^{\star})).$$
By Lemma \ref{stabilizer}, if $\chi^{[g,h]}(x\cdot h(x^{\star}))\neq 0$ for some $x\in P_{g}$ then we have $\chi\in\Sim(A_{1})^{<g,hg^{-1}h^{-1}>}$.
Otherwise each individual term in the summation and hence $\chi^{[g,h]}(\Omega_{g,h})$ must be zero.

Let  us assume that $\chi\in\Sim(A_{1})^{<g,hg^{-1}h^{-1}>}$. Then in this case  we get
\begin{equation}\label{eq:omega}
    \begin{split}
\chi^{[g,h]}(\Omega_{g,h})
&=\sum_{x\in P_{g}}\frac{\chi^{g}(x)\cdot\chi^{hg^{-1}h^{-1}}(h(x^{\star}))}{\chi\circ\phi(g,hg^{-1}h^{-1})}\\
&=\sum_{x\in P_{g}}\frac{\chi^{g}(x)\cdot(\chi^{hg^{-1}h^{-1}}\circ h(x^{\star}))}{\chi\circ\phi(g,hg^{-1}h^{-1})}.
\end{split}
\end{equation}
Recall that here $\chi^{hg^{-1}h^{-1}}:A_{hg^{-1}h^{-1}}\rightarrow\C$ is the restriction of some choice of character 
$\widetilde{\chi^{hg^{-1}h^{-1}}}:A_{<hg^{-1}h^{-1}>}\rightarrow\C$ extending $\chi:A_{1}\rightarrow \C$.
Hence $\widetilde{\chi^{hg^{-1}h^{-1}}\circ h}:A_{<g^{-1}>}\rightarrow\C$ is a character extending $\chi\circ h:A_{1}\rightarrow\C$.
So the composition $\chi^{hg^{-1}h^{-1}}\circ h:A_{g^{-1}}\rightarrow\C$ differ from the choosen twisted character $(\chi\circ h)^{g^{-1}}$ by some $o(g)^{th}$-root of unity.

 By the orthogonality of twisted characters, we get that $\chi^{[g,h]}(\Omega_{g,h})=0$ if $\chi\neq\chi\circ h$.
In other words, we have proved that if $\chi^{[g,h]}(\Omega_{g,h})\neq0$ then we must have $\chi\in\Sim(A_{1})^{<g,h>}$.

Let us assume that $g,h\in G_{\chi}$ then by twisted orthogonality relations we have 
$$\sum_{x\in P_{g}}\chi^{g}(x)\overline{\chi^{g}(x)}=\sum_{x\in P_{g}}\chi^{g}(x)\chi^{g^{-1}}(x^{\star})=\chi(\alpha_{\chi})$$
and  $$\chi^{hg^{-1}h^{-1}}\circ h=\frac{\chi^{g^{-1}}}{\chi\circ\phi(h,g^{-1},h^{-1})}.$$
Thus combining this with equation (\ref{eq:omega}) we get that
$$\chi^{[g,h]}(\Omega_{g,h})=\frac{\chi(\alpha_{\chi})}{\chi\circ\phi(g,h,g^{-1},h^{-1})}.$$
This prove the lemma.
\end{proof}
Now using the Theorem \ref{Verlinde formula:genus 0} and Lemma \ref{lemma:verlinde} we obtain the following:
\begin{corollary}[twisted Verlinde formula for any genus]\label{Verlinde formula:any genus}
Let $r,s$ be non-negative integers and let $g_{1},g_{2},\cdots,\\
g_{r},h_{1},h_{2},\cdots,h_{r},m_{1},\cdots,m_{s}\in G$ be any elements such that
$[g_{1},h_{1}]\cdots [g_{r},h_{r}]\cdot m_{1}\cdots m_{s}=1$.
Let $G^{\circ}\leq G$ be a subgroup of $G$ generated by
$g_{i},h_{i},m_{j}$.
Let $a_{j}\in A_{m_{j}}$ then 
$$\lambda(\Omega_{g_{1},h_{1}}\cdots\Omega_{g_{r},h_{r}}\cdot a_{1}\cdots a_{s})=\sum_{\chi\in\Sim(A_{1})^{G^{\circ}}}\frac{\chi(\alpha_{\chi})^{r-1}\chi^{m_{1}}(a_{1})\cdots\chi^{m_{s}}(a_{s})}{\chi\circ\phi(g_{1},h_{1},g_{1}^{-1},h_{1}^{-1},\cdots, m_{1},\cdots, m_{s})}.$$
\end{corollary}

\subsection{Examples coming from fusion categories }\label{sec:motivation}
In this subsection we will see that an important class examples of $G$-crossed Frobenius $\s$-algebra arises from the theory of fusion categories.
We will begin by recalling some definition and results from  \cite{DGNO},\cite{Onfusioncategories}, \cite{FusionCA}, \cite{Turaev1999HomotopyFT}.

A \textit{tensor category} (or \textit{monoidal $\C$-linear category}) is a $\C$-linear abelian category $\mathcal{C}$ with a structure of monoidal category such that
the bifunctor $\otimes:\mathcal{C}\times\mathcal{C}\rightarrow\mathcal{C}$ is bilinear on morphisms.
By a \textit{fusion category} we mean a $\C$-linear semisimple rigid tensor category $\mathcal{C}$ with finitely many isomorphism classes of simple objects, finite-dimensional spaces of morphisms, and such that the unit object $\mathbbm{1}$ of $\mathcal{C}$ is simple.
 A fusion category is \textit{braided} if it is equipped with a natural isomorphism $c_{X,Y}:X\otimes Y\xrightarrow{\sim}Y\otimes X$ satisfying the hexagon axioms.

\begin{remark}
Let $\mathcal{C}$ be a fusion category. The tensor product on $\mathcal{C}$ induces the ring structure on Grothendieck ring.
We will denote complexified Grothendieck ring of $\mathcal{C}$ by $K(\mathcal{C})$.
 We have the non-degenerate symmetric linear functional
$\lambda :K(\mathcal{C})\rightarrow \C$ which represents the coefficient of the class of the unit $[\mathbbm{1}]$ in any element of $K(\mathcal{C})$.
The duality in $\mathcal{C}$ induces a $\star$-map on
 $K(\mathcal{C})$, which is defined to be semilinear.
 Thus $K(\mathcal{C})$ becomes a Frobenius $\s$-algebra.
Moreover, if $\mathcal{B}$ is braided fusion category then $K(\mathcal{B})$ is a commutative Frobenius $\star$-algebra.
\end{remark}

\begin{definition}
 Let $G$ be a finite group. A $G$-\textit{grading} on a tensor category $\mathcal{C}$ is a decomposition
$$\mathcal{C}=\bigoplus_{g\in G}\mathcal{C}_{g}$$
into a direct sum of full abelian subcategories such that tensor product $\otimes$ maps $\mathcal{C}_{g}\times\mathcal{C}_{h}$ to $\mathcal{C}_{gh}$ for all $g,h \in G$.
In this case, trivial component  $\mathcal{C}_{1}$ is a full monoidal subcategory and each $\mathcal{C}_{g}$ is $\mathcal{C}_{1}$-bimodule category.
\end{definition}

\begin{remark}
Let $\mathcal{C}$ be a fusion category,
and $\mathcal{M}$ a right $\mathcal{C}$-module category. 
Let $\mathcal{M}^{op}$  be the category opposite to $\mathcal{M}$. 
Then $\mathcal{M}^{op}$ is a left $\mathcal{C}$-module category with the $\mathcal{C}$-action $\odot$ given
by $X\odot M \coloneqq M\otimes\prescript{\s}{}{X}$.
Similarly, if $\mathcal{N}$ is a left $\mathcal{C}$-module category,
then $\mathcal{N}^{op}$ is a right $\mathcal{C}$-module category, with the $\mathcal{C}$-action $\cdot$ given by
$N\cdot X\coloneqq X^{\s}\otimes N$.
Note that $(\mathcal{M}^{op})^{op}$ is canonically equivalent to $\mathcal{M}$
as a $\mathcal{C}$-module category.
If $\mathcal{M}$ is a $\mathcal{C}$-(left)module category then $K(\mathcal{M})$ is a $K(\mathcal{C})$-(left)module.
\end{remark}

\begin{definition}
A $(\mathcal{C},\mathcal{D})$-bimodule category $\mathcal{M}$ is said to be  \textit{invertible} if there exist bimodule equivalences:
 $$\mathcal{M}\boxtimes_{\mathcal{C}}\mathcal{M}^{op}\cong \mathcal{D}\ \mbox{and}\ \mathcal{M}^{op}\boxtimes_{\mathcal{D}}\mathcal{M}\cong \mathcal{C}.$$
\end{definition}

Let $\B$ be a braided fusion category, and $\mathcal{M}$ be a left $\B$-module category. Then $\mathcal{M}$ is automatically a $\B$-bimodule category(using the braiding on $\B$).
We will say a $\B$-module category is invertible if it is invertible as a $\B$-bimodule category.
Let $\Pic(\B)$ denotes the group of \textit{invertible} $\B$-module categories and $\underline{\Pic}(\B)$ denotes the categorical $1$-group whose objects are invertible $\B$-module categories and $1$-morphisms are equivalences of $\B$-module categories, see \cite{FusionCA}.
Now we will state some results without proof.
\begin{lemma}\cite[\textsection 8.4]{FusionCA}\label{asso:category}
Let $G$ be a finite group and $\B$ be a non-degenerate braided fusion category.
Let $c:G\rightarrow\Pic(\B)$ be a group homomorphism $i.e.$ for each $g\in G$ there is an invertible $\B$-module category $\mathcal{C}_{g}$ and for each pair $g,h\in G$ there is an equivalence of $\B$-module categories $M_{g,h}:\mathcal{C}_{g}\boxtimes_{B}\mathcal{C}_{h}\cong\mathcal{C}_{gh}$.
Then the following diagram:
\begin{center}
\begin{tikzcd}[column sep = 2cm]
{\mathcal{C}_{f}\boxtimes_{\B} \mathcal{C}_{g}\boxtimes_{\B}\mathcal{C}_{h}}\arrow[r, "M_{f,g}\boxtimes Id_{\mathcal{C}_{h}}"]\arrow[d, swap, "Id_{\mathcal{C}_{f}}\boxtimes M_{g,h}"]&
{\mathcal{C}_{fg}\boxtimes_{\B} \mathcal{C}_{h}}\arrow[d, "M_{fg,h}"] \\
{\mathcal{C}_{f}\boxtimes_{\B} \mathcal{C}_{gh}}\arrow[r, "M_{f,gh}"]&
\mathcal{C}_{fgh}
\end{tikzcd}
\end{center}
commutes upto an element $T\in H^{3}(G,\mathcal{O}_{\B}^{\times})$, where $\mathcal{O}_{\B}^{\times}$ denotes the abelian group of isomorphism classes of invertible objects in $\B$.
\end{lemma}
\begin{remark} \cite[\textsection 5]{FusionCA}\label{pic and eqbr}
Let $\B$ be a non-degenerate braided fusion category then $\underline{\Pic}(\B)\cong\underline{\mbox{EqBr}}(\B)$, where $\underline{\mbox{EqBr}}(\B)$ denotes a categorical $1$-group whose objects are braided autoequivalences of $\B$ and morphisms are natural  isomorphisms between braided autoequivalences.
\end{remark}

\begin{remark}\label{Group cohomology} We recall some results from  group cohomology.\\
(1) Let $M$ be a $G$-module.
The maps
$$\Psi^{i}:\Hom_{\ZG}(\mathbb{Z}[G^{i+1}],M)\rightarrow C^{i}(G,M)$$
defined by 
$$\Psi^{i}(\phi)(g_{1},g_{2},\cdots,g_{i})=\phi(1,g_{1},g_{1}g_{2},\cdots,g_{1}g_{2}\cdots g_{i})$$
are isomorphisms for all $i\geq 0$.
This provides isomorphisms between  complexes
$$\cdots\xrightarrow[]{d^{i-1}}\Hom_{\ZG}(\mathbb{Z}[G^{i}],M)\xrightarrow[]{d^{i}}\Hom_{\ZG}(\mathbb{Z}[G^{i+1}],M)\xrightarrow[]{d^{i+1}}\cdots$$
and 
$$\cdots\xrightarrow[]{d^{i-1}}C^{i}(G,M)\xrightarrow[]{d^{i}}C^{i+1}(G,M)\xrightarrow[]{d^{i+1}}$$
in the sense that $\Psi^{i+1}\circ d^{i+1}=d^{i}\circ\Psi^{i}$ for all $i\geq 0$.
Moreover, these isomorphisms are natural in the $G$-module $M$.\\
(2) Let $H$ be a subgroup of $G$ and let $M$ be a $H$-module.
The maps 
$$\Phi_{i}:\Hom_{\ZG}(Z[G^{i}],\CoInd_{H}^{G}(M))\rightarrow \Hom_{\mathbb{Z}[H]}(\mathbb{Z}[G^{i}],M)$$
defined by $$\Phi_{i}(\phi)(x)=\phi(x)(1)$$ are  isomorphisms for all $i\geq 0$.\\
(3) Let $H$ be a subgroup of $G$ then $\ZG$ is a free $\mathbb{Z}[H]$-module with basis parametrise by right coset representatives of $H$ in $G$.
Consider
\begin{center}
\begin{tikzcd}
\cdots\arrow[r]&
\mathbb{Z}[G^{i+1}]\arrow[r,"d^{i}"]\arrow[d,"\alpha_{i+1}"]&
\mathbb{Z}[G^{i}]\arrow[r]\arrow[d,"\alpha_{i}"]&
\cdots\arrow[r]&
\ZG\arrow[r,"\epsilon"]\arrow[d,"\alpha_{1}"]&
\mathbb{Z}\arrow[r]\arrow[d,"Id"]&
0\\
\cdots\arrow[r]&
\mathbb{Z}[H^{i+1}]\arrow[r,"d^{i}"]&
\mathbb{Z}[H^{i}]\arrow[r]&
\cdots\arrow[r]&
\mathbb{Z}[H]\arrow[r,"\epsilon"]&
\mathbb{Z}\arrow[r]&
0
\end{tikzcd}
\end{center}
where 
$$d^{i}(g_{0},g_{1},\cdots,g_{i})=\sum_{j=0}^{i}(-1)^{j}(g_{0},\cdots,g_{j-1},g_{j+1},\cdots,g_{i})$$
for all $i\geq1$,$\epsilon$ is the augmentation map.
Then there exists $\alpha_{i}$   a $\mathbb{Z}[H]$-module map such that $\epsilon\circ\alpha_{1}=\epsilon$ and  $\alpha_{i}\circ d^{i}=d^{i}\circ\alpha_{i+1}$ for all $i\geq 1$.\\
(4) Suppose that $0\rightarrow L\rightarrow M\rightarrow N\rightarrow 0$ is a short exact sequence of $G$-modules. 
Then there exists connecting homomorphisms $\delta^{*}:H^{i}(G,L)\rightarrow H^{i+1}(G,N)$  and a long exact sequence of abelian groups
$$0\rightarrow H^{0}(G,L)\rightarrow H^{0}(G,M)\rightarrow H^{0}(G,N)\xrightarrow{\delta^{0}} H^{1}(G,L)\rightarrow \cdots$$
Moreover, this construction is natural in the short exact sequence in the sense that any morphism of
short exact sequences gives a morphism of long exact sequences.
\end{remark}

\begin{remark}
Let $\B$ be a non-degenerate braided fusion category and let $c:G\rightarrow \Pic(\B)\cong \mbox{EqBr}(\B)$ be a group homomorphism. 
Then we have an action  of $G$ on $K(\B)$ by Frobenius $\star$-algebra automorphism.
This induces an action of $G$ on $\Sim(K(\B))$.
By \cite[\textsection 2]{Oncenters}, $K(\mathcal{C}_{g^{-1}})$ has basis  indexed by $\Sim(K(\B))^{g}=\{\chi\in\Sim(K(\B)):\prescript{g}{}{}\chi=\chi\}$. 
Moreover, the decomposition of $K(\mathcal{C}_{g^{-1}})$ into simple  $K(\B)$-module  is given by $$K(\mathcal{C}_{g^{-1}})\cong\bigoplus_{\chi\in\Sim(K(\B))^{g}}\C\cdot e_{\chi}^{g}$$
as a $K(\B)$-module.
For $\chi\in\Sim(K(\B))$, $G_{\chi}$ denotes the stabilizer of $\chi$ in $G$.
\end{remark}
Now using the definition of a tensor product of $\B$-module categories, we have  a bilinear map  on $K(\D)=\bigoplus_{g\in G}K(\mathcal{C}_{g})$ such that $K(\mathcal{C}_{g})\times K(\mathcal{C}_{h})$ maps to $K(\mathcal{C}_{gh})$.
By Lemma \ref{asso:category}, this map satisfies associativity property up to an element $T\in 
H^{3}(G,\mathcal{O}_{\B}^{\times})$.

As $\mathcal{O}_{\B}^{\times}$ is a $G$-stable subgroup of $K(\B)^{\times}$, so
we have a short exact sequence of $G$-modules,
\begin{equation}\label{eq: ses}
    0\rightarrow \mathcal{O}_{\B}^{\times}\xrightarrow{i} K(\B)^{\times}\xrightarrow{\pi} K(\B)^{\times}/\mathcal{O}_{\B}^{\times}\rightarrow 0.
\end{equation}
Using this short exact sequence we get a long exact sequence,
$$\cdots\rightarrow  H^{2}(G,K(\B)^{\times})\xrightarrow{\pi^{*}} H^{2}(G,K(\B)^{\times}/\mathcal{O}_{\B}^{\times})\xrightarrow{\delta} H^{3}(G,\mathcal{O}_{\B}^{\times})\xrightarrow{i^{*}} H^{3}(G,\mathcal{O}_{\B}^{\times})\rightarrow \cdots.$$
We will prove that $i^{*}(T)$ is the trivial element in $H^{3}(G,K(\B)^{\times})$.

\begin{prop}\label{prop: consrt t}
Let $c:G\rightarrow\Pic(\B)$ be a group homomorphism
and $K(\mathcal{D})=\oplus_{g\in G}K(\mathcal{C}_{g})$ be the corresponding $G$-graded $K(\B)$-bimodule.
Then there exists a $2$-cochain $\phi\in C^{2}(G,K(\B)^{\times})$ such that, 
$$\mbox{for}\ g,h\in G_{\chi},\ \ \ e_{\chi}^{g}\cdot e_{\chi}^{h}=\chi\circ\phi(g,h)e_{\chi}^{gh}.$$
Moreover, $\phi$ satisfies the $2$-cocycle condition modulo $\mathcal{O}_{\B}^{\times}$ \emph{i.e} $t\coloneqq\pi^{*}(\phi)\in H^{2}(G,K(\B)^{\times}/\mathcal{O}_{\B}^{\times})$, where $\pi$ as in (\ref{eq: ses}). In particular, $\phi$ satisfies the following relation:
$$\phi(fg,h)\cdot\phi(f,g)=T_{f,g,h}\cdot\phi(f,gh)\cdot\phi(g,h)\ \ \ \ \mbox{for all}\ f,g,h\in G.$$
\end{prop}
\begin{proof}
For $g,h\in G_{\chi}$, we have $e_{\chi}^{g}\cdot e_{\chi}^{h}\in K(\mathcal{C}_{gh})$.
But the $K(\B)$-module structure on $K(\mathcal{C}_{gh})$ implies that there exists a $\phi_{\chi}\in C^{2}(G_{\chi},\C^{\times})$ such that
$$e_{\chi}^{g}\cdot e_{\chi}^{h}=\phi_{\chi}(g,h)e_{\chi}^{gh}.$$
By Lemma \ref{asso:category},   $\phi_{\chi}$ satisfies the  following relation: 
\begin{equation}\label{eq: for phi chi}
    \mbox{for}\ f,g,h\in G_{\chi},\ \ \ \phi_{\chi}(f,g)\phi_{\chi}(fg,h)=\chi(T_{f,g,h})\phi_{\chi}(f,gh)\phi_{\chi}(g,h).
\end{equation}
By Remark \ref{Group cohomology}, we get an element $\psi_{\chi}\in C^{2}(G,\CoInd_{G_{\chi}}^{G}(\C^{\times}))$.
We have $K(\B)^{\times}=\oplus\CoInd_{G_{\chi}}^{G}(\C^{\times})$, where the sum runs over the orbit representatives for the induced action of $G$ on $\Sim(K(\B))$.
Using this we can get an element $\phi\in C^{2}(G,K(\B)^{\times})$ such that for $\chi\in \Sim(K(\B))$
$$\chi(\phi(g,h))=\phi_{\chi}(g,h),\ \ \ \mbox{for all}\ g,h\in G_{\chi}.$$
This proves the first part of proposition.

Let us define an element $t\in C^{2}(G,K(\B)^{\times}/\mathcal{O}_{\B}^{\times})$ as follows:
$$t\coloneqq\pi^{*}(\phi)=\pi\circ\phi.$$
Using the (\ref{eq: for phi chi}), we get, 
$$ \phi(fg,h)\cdot\phi(f,g)=T_{f,g,h}\cdot\phi(f,gh)\cdot\phi(g,h)\ \ \ \ \mbox{for all}\ f,g,h\in G.$$
This implies that $\phi$ satisfies $2$-cocycle condition modulo $\mathcal{O}_{\B}^{\times}$
and $t$ is a $2$-cocycle.
This proves the proposition.
\end{proof}

\begin{corollary}\label{associativity}
Let $\delta:H^{2}(G,K(\B)^{\times}/\mathcal{O}_{\B}^{\times})\rightarrow H^{3}(G,\mathcal{O}_{\B}^{\times})$ be the connecting homomorphism associated with (\ref{eq: ses}).
Then $\delta(t)=T$, where $T$ as in Lemma \ref{asso:category}.
Hence, $i^{*}(T)$ is the trivial element in $H^{3}(G,K(\B)^{\times})$. 
\end{corollary}
\begin{proof}
By Prop. \ref{prop: consrt t} and by the definition of  connecting homomorphism $\delta$, we have $$\delta(t)=T \in H^{3}(G,\mathcal{O}_{\B}^{\times}).$$
Hence, $$i^{*}(T)=i^{*}(\delta(t))=i^{*}\circ\delta(t)=0\in H^{3}(G,K(\B)^{\times}).$$
This proves the result.
\end{proof}
\begin{remark}\label{remark: torsor and T}
Let $c:G\rightarrow\Pic(\B)$ be a group homomorphism.
Then \cite[Theorem 8.5]{FusionCA} implies that the set of all extensions of $c$ to a morphism of $1$-groups  $\underline{c}:G\rightarrow\underline{\Pic}(\B)$ forms a torsor over $H^{2}(G,\mathcal{O}_{\B}^{\times})$.
Moreover this torsor is non empty iff the obstruction $T\in H^{3}(G,\mathcal{O}_{\B}^{\times})$ as in Lemma \ref{asso:category}  vanishes.
Let us denote this torsor by $\mathcal{T}$.
\end{remark}
\begin{corollary}\label{Coro: troser to cohomology}
Let us assume that the torsor $\mathcal{T}$ is non-empty.
 Then there is a mapping $\Phi:\mathcal{T}\rightarrow H^{2}(G,K(\B)^{\times})$ such that $\pi^{*}(\Phi(\mathcal{T}))=t\in H^{2}(G,K(\B)^{\times}/O_{\B}^{\times})$.
 In particular, we can identify the torsor $\mathcal{T}/\mbox{ker}(i^{*})$ with $(\pi^{*})^{-1}(t)\subseteq H^{2}(G,K(\B)^{\times})$ as a torsor over $H^{2}(G,\mathcal{O}_{\B}^{\times})/\mbox{ker}(i^{*})$.
  \end{corollary}
 \begin{proof}
 The elements of the torsor $\mathcal{T}$ are nothing but the isomorphism classes of $\B$-bimodule tensor products on $\D=\oplus_{g\in G}\mathcal{C}_{g}$.
 Choose $\B$-bimodule equivalences $M_{g,h}:\mathcal{C}_{g}\boxtimes_{\B}\mathcal{C}_{h}\rightarrow \mathcal{C}_{gh}$ for all $g,h\in G$.
 By Remark \ref{remark: torsor and T}, the obstruction $T\in H^{3}(G,\mathcal{O}_{\B}^{\times})$ is vanishes.
 Then by Lemma \ref{asso:category}, the $\B$-bimodule equivalences $\{M_{g,h}\}_{g,h\in G}$ defines an element $M$ of  $\mathcal{T}$.
 Using the Prop. \ref{prop: consrt t}, we get a $2$-cocycle $\phi\in H^{2}(G, K(\B)^{\times})$ associated with $M$.
Note that if we change the element $M\in \mathcal{T}$ by  $\psi\in H^{2}(G,\mathcal{O}^{\times})$ \emph{i.e.}
 replace $M_{g,h}$ by $\psi_{g,h}M_{g,h}$.
 Then the corresponding element $\phi$ will be replaced by $\psi\cdot\phi$ in $H^{2}(G,K(\B)^{\times})$.
 Then using
 $$\mathcal{T}\ni M\mapsto \phi\in H^{2}(G,K(\B)^{\times})$$
 we get a well defined  morphism $\Phi:\mathcal{T}\rightarrow H^{2}(G,K(\B)^{\times})$ of $H^{2}(G,\mathcal{O}_{\B}^{\times})$-torsors. 
 By the definition of $t$ (as in Prop \ref{prop: consrt t}),
 $\pi^{*}(\Phi(\mathcal{T}))=t$.
 This proves the result.
 \end{proof}

\begin{remark}\label{pic:G-graded algebra}
Let $\underline{c}:G\rightarrow \underline{\Pic}(\B)$ be a morphism of $1$-groups which lifts $c$. Let $\D$ denote the corresponding $G$-graded $\B$-bimodule category, $\D=\oplus_{g\in G}\mathcal{C}_{g}$. 
Then the obstruction $T\in H^{3}(G,\mathcal{O}_{\B}^{\times})$ as in Lemma \ref{asso:category} is vanishes and we get a $\B$-bimodule tensor product on $\D$ (cf.\cite[Theorem 8.4]{FusionCA}). 
By Prop. \ref{prop: consrt t},  $K(\D)$ becomes a $G$-graded associative unital $\C$-algebra with trivial  component equal to $K(\B)$.
Note that $K(\B)$ is a Frobenius $\star$-algebra, so we have a linear functional $\lambda_{0}:K(\B)\rightarrow\C$, extend it linearly  to $K(\D)$ by zero outside of $K(\B)$ to get a linear functional on $K(\D)$.
Using Remark \ref{star} and Prop. \ref{Star:unique}, one can define a  $\star$-map on $K(\D)$.
Thus $K(\D)$ becomes a $G$-graded Frobenius $\star$-algebra.
\end{remark}
\begin{remark}\label{star}(see \cite[\textsection 2.1 and \textsection 2.4]{Modularcatcrossedmatrix})
Let $\B$ be a non-degenerate braided fusion category and $\mathcal{M}$ be an invertible $\B$-module category.
Then there exists  a $\Z/N\Z$-graded  fusion category  
$\D=\oplus_{a\in \Z/N\Z}\mathcal{C}_{a}$
such that the trivial component is equal to $\B$ and $\mathcal{C}_{\bar{1}}\cong\mathcal{M}$ as a $\B$-bimodule category.
The complexified Grothendieck ring $K(\D)$ is a $\Z/N\Z$-graded Frobenius $\star$-algebra.
The $\star$-maps $K(\mathcal{M})$ to $K(\mathcal{C}_{N-1})\cong K(\mathcal{M}^{op})$.
\end{remark}

\begin{prop}\label{Thm:pic crossed algebra}
Let $\underline{c}:G\rightarrow \underline{\Pic}(\B)$ be a morphism of $1$-groups and $\D$ denotes the corresponding $G$-graded $\B$-bimodule category, $\D=\oplus_{g\in G}\mathcal{C}_{g}$. 
Then $K(\mathcal{D})$ is a  strict $G$-crossed Frobenius $\star$-algebra with trivial component $K(\B)$.
\end{prop}
\begin{proof}
By Remark \ref{pic:G-graded algebra},
$K(\D)$ is a $G$-graded Frobenius $\star$-algebra with
 trivial component $K(\B)$.
 To prove the theorem we have to check only  for the $G$-action.

Indeed, for all $g, h \in G$ the category $\mbox{Fun}_{\B}(\mathcal{C}_{g}, \mathcal{C}_{gh})$ of $\B$-module functors from $\mathcal{C}_{g}$ to $\mathcal{C}_{gh}$ is identified,
on one hand, with functors of right tensor multiplication by objects of
$\mathcal{C}_{h}$ and, on the other hand, with functors of left tensor 
multiplication by objects of $\mathcal{C}_{ghg^{-1}}$. So there is an 
equivalence $g :\mathcal{C}_{h}\rightarrow \mathcal{C}_{ghg^{-1}}$ defined
by the isomorphism of $\B$-module functors
\begin{equation}\label{action:equation}
-\otimes Y\cong g(Y)\otimes -:\mathcal{C}_{g}\rightarrow\mathcal{C}_{gh},\ \ \ Y\in\mathcal{C}_{h}.
\end{equation}
Extending it to $\D=\oplus_{g\in G}\mathcal{C}_{g}$ by linearity, we obtain an action of $G$ by  autoequivalences of $\mathcal{D}$. Furthermore, evaluating (\ref{action:equation}) on $X\in\mathcal{C}_{g}$ we
obtain a natural family of isomorphisms
\begin{equation}
    X\otimes Y\xrightarrow{\sim} g(Y)\otimes X,\ \ \ \ g\in G,\ X\in\mathcal{C}_{g},Y\in \mathcal{D}
\end{equation}
which gives an action of $G$ on $K(\D)$ by algebra automorphism. 
As each $g\in G$ acts on $\D$ by autoequivalence and on trivial component it acts by tensor autoequivalence.
So $\dim(\Hom(\mathbbm{1},X))=\dim(\Hom(\mathbbm{1},g(X))$ for all $X\in \B$.
Note that the dual object is unique up to a unique isomorphism.
This implies that the $G$-action commute with the $\star$-map.
Also the fix points for the  partial action is equal to the fix points for the above $G$-action is follows from  \cite[Lemma 3.5]{Oncenters}.
This prove that $K(\D)$ is a strict $G$-crossed extension of $K(\B)$.
\end{proof}
\begin{remark}\label{rk:cohclass}
Let $[\phi]\in H^{2}(G,K(\B)^{\times})$ be the cohomology class corresponding to the $G$-crossed extension $K(\D)$ of $K(\B)$.
Then we have $\pi^{*}(\phi)=t$. Moreover the equivalence class of the lift $\underline{c}:G\rightarrow \underline{\Pic}(\B)$ can be considered an element  $[\underline{c}]$ in the $H^2(G,\mathcal{O}^\times_{\B})$-torsor $\mathcal{T}$ and we have $[\phi]=\Phi([\underline{c}])$. Here $\Phi:\mathcal{T}\to  H^{2}(G,K(\B)^{\times})$ is the mapping from Corollary \ref{Coro: troser to cohomology}.
\end{remark}
\begin{corollary}\label{coro: twisted Verlinde formula}(Twisted Verlinde formula.)
Let $r,s$ be non-negative integers and let $g_1,h_1,\cdots, g_r,h_r$,  $m_{1}, \cdots, m_{s} \in G$ be any elements such that
$[g_{1},h_{1}]\cdots [g_{r},h_{r}]\cdot m_{1}\cdots m_{s}=1$.
Let $G^{\circ}\leq G$ be a subgroup of $G$ generated by 
$g_{i},h_{i},m_{j}$.  For $g\in G$, let $P_{g}$ denote the set of isomorphism classes of simple objects in $\mathcal{C}_{g}$ and define the object  $$\Omega_{g_{i},h_{i}}:=\bigoplus_{X\in P_{g_{i}}} X\otimes h_{i}(X^{\star})\in \mathcal{C}_{[g_i,h_i]}.$$
Let $M_{j}\in \mathcal{C}_{m_{j}}$. Then 
$$\dim(\Hom(\mathbbm{1},\Omega_{g_{1},h_{1}}\otimes\cdots\otimes\Omega_{g_{r},h_{r}}\otimes M_{1}\otimes\cdots \otimes M_{s}))=\sum_{\chi\in\Sim(K(\B))^{G^{\circ}}}\frac{\chi(\alpha_{\chi})^{r-1}\chi^{m_{1}}([M_{1}])\cdots\chi^{m_{s}}([M_{s}])}{\chi\circ\phi(g_{1},h_{1},g_{1}^{-1},h_{1}^{-1},\cdots, m_{1},\cdots, m_{s})}.$$
\end{corollary}
\begin{proof}
This is a special case of the Corollary \ref{Verlinde formula:any genus}.
\end{proof}
\begin{definition}[Modular Category]
A modular category $\B$ is a braided fusion category with a ribbon twist such that the corresponding $S$-matrix is invertible.
Where  a ribbon twist on a braided fusion category $\B$ is a natural isomorphism $\theta:id_{\B}\rightarrow id_{\B}$ such that $\theta_{C^{\star}}=\theta_{C}^{\star}$  for all $C\in\B$ and 
$$\theta_{C_{1}\otimes C_{2}}=(\theta_{C_{1}}\otimes\theta_{C_{2}})\circ\beta_{C_{2},C_{1}}\circ\beta_{C_{1},C_{2}}\ \ \ \mbox{for any}\ C_{1},\ C_{2}\in \B$$ 
and the unnormalized $S$-matrix, $S$ is an $\mathcal{O}_{\B}\times\mathcal{O}_{\B}$ matrix defined as $$S_{C_{1},C_{2}}\coloneqq \trace(\beta_{C_{2},C_{1}}\circ\beta_{C_{1},C_{2}})\ \mbox{ for}\ C_{1},C_{2}\in \mathcal{O}_{\B}.$$
\end{definition}
\begin{remark}
Using the unnormalized $S$-matrix  $S$ associated with the modular category $\B$, we can identify $P_{1}$ with $\Sim(K(\B))$ as follows:
 $$\mbox{for}\ P_{1}\ni C\mapsto \left(\chi_{C}:[D]\mapsto\frac{S(\B)_{D,C}}{\dim(C)}\right)$$
 \textit{i.e.} the $S$-matrix is essentially the character table of $K(\B)$.
\end{remark}
\begin{remark}
Let $\B$ be a modular category. 
Let $\underline{\EqMod}(\B)$ denote the categorical $1$-group whose objects are modular autoequivalences of $\B$ and morphisms are natural isomorphisms between modular autoequivalences.
One can see that $\underline{\EqMod}(\B)$ is a full $1$-subgroup of the categorical
$1$-group $\underline{\mbox{EqBr}}(\B)$ of all braided autoequivalences of $\B$.
Let $\underline{\underline{\Pic}}^{\mbox{tr}}(\B)\subseteq\underline{\underline{\Pic}}(\B)$ be the full $2$-subgroup formed by those invertible $\B$-module categories which can be equipped with a $\B$-module trace (see \cite{Schaumann}, \cite[\S1.3]{Modularcatcrossedmatrix}). Similarly, we have the full 1-subgroup ${\underline{\Pic}}^{\mbox{tr}}(\B)\subseteq{\underline{\Pic}}(\B).$
From Remark \ref{pic and eqbr} and using \cite[\S1.3]{Modularcatcrossedmatrix}, we have an equivalence of 1-groups $$\underline{\EqMod}(\B)\cong\underline{\Pic}^{\mbox{tr}}(\B).$$
A $\B$-module trace $\trace_{\M}$ on $\M$ assigns a trace $\trace_{\M}(f)\in\C$ for each endomorphism $f:M\rightarrow M$ in $\M$ satisfying some properties and in particular
 we can talk of dimensions $\dim_{\M}(M)$ of objects of $M$.
 We will often assume that the trace is normalized in such a way that
 $\dim(\underline{\mbox{Hom}}(M,N))=\dim_{\M}(M)\cdot\dim_{\M}(N)$ for $M,N\in\M$,
 where $\underline{\mbox{Hom}}(M,N)\in\B$ is the internal Hom.
 With this additional condition, $\trace_{\M}$ is uniquely defined upto scaling by $\pm 1$
 and we have 
 $$\sum_{M\in\mathcal{O}_{\mathcal{M}}}\dim_{\M}(M)^{2}=\dim(\B)=\sum_{C\in \mathcal{O}_{\B}}\dim(C)^{2},$$
 where $\mathcal{O}_{\M}$ denotes the set of simple objects of $\M$.
 Moreover with such a normalization, $\dim_{\M}(M)$ must be a totally real cyclotomic integer for each $M\in\M$.
 For more details we refer to \cite{Modularcatcrossedmatrix}, \cite{DGNO}.
 \end{remark}
 
 Let $\underline{c}:G\rightarrow\underline{\Pic}^{\mbox{tr}}(\B)\cong\underline{\mbox{EqMod}}(\B)$ be a morphism of $1$-groups, i.e. a modular action of $G$ on $\B$.
 Then $\D=\oplus_{g\in G}\mathcal{C}_{g}$ is a $G$-graded category equipped with a modular action of $G$ on $\mathcal{C}_{1}(=\B)$ and each $\mathcal{C}_{g}$ is equipped with a $\mathcal{C}_{1}$-module trace.
 Suppose   $P_{g}$ denotes the set of isomorphism classes of simple objects in $\mathcal{C}_{g}$. 
 In this setting we can define a $P_{g}\times P_{1}^{g}$-matrix $S^{g}$ known as the unnormalized $g$-crossed $S$-matrix as follows:
 \begin{definition}
 For each simple object $C\in P^{g}_{1}$,
 let us choose an isomorphism $\psi_{C}:g(C)\rightarrow C$ such that the induced composition
 $$C=g^{m}(C)\rightarrow g^{m-1}(C)\rightarrow\cdots\rightarrow g(C)\rightarrow C$$ is the identity,
 where $m$ is the order of $g$ in G.
 Note that the above conditions implies that $\psi_{C}$ is well-defined up to scaling by $m^{th}$-root of unity.
 For each simple object $M\in P_{g}$ and $C\in P_{1}^{g}$,
 $$S^{g}_{M,C}\coloneqq \trace_{\mathcal{C}_{g}}(C\otimes M\xrightarrow{\beta_{C,M}}M\otimes C\xrightarrow{\beta_{M,C}}g(C)\otimes M\xrightarrow{\psi_{C}\otimes id}C\otimes M).$$
 \end{definition}
 \begin{remark}
 We will always choose $\psi_{\mathbbm{1}}:\mathbbm{1}\rightarrow\mathbbm{1}$ to be the identity.
 With this convention we see that $S(\B,g)_{M,\mathbbm{1}}=\trace_{\mathcal{C}_{g}}(id_{M})=\dim_{\mathcal{C}_{g}}(M)$,
 the categorical dimension of the object $M\in\mathcal{C}_{g}$.
  \end{remark}
  
  Let us choose   a normalized $\B$-module trace on each invertible $\B$-bimodule category $\mathcal{C}_{g}$ for all $g\in G$.
  For each $g,h\in G$, this induces a normalized $\B$-module trace $\trace_{\mathcal{C}_{g}\boxtimes_{\B}\mathcal{C}_{h}}$ on the invertible $\B$-bimodule category $\mathcal{C}_{g}\boxtimes_{\B}\mathcal{C}_{h}$.
  We have $\B$-module equivalence $\mathcal{C}_{g}\boxtimes_{\B}\mathcal{C}_{h}\cong\mathcal{C}_{gh}$.
  The normalized trace $\trace_{\mathcal{C}_{g}\boxtimes_{\B}\mathcal{C}_{h}}$
  on $\mathcal{C}_{g}\boxtimes_{\B}\mathcal{C}_{h}$ either agrees with the normalized $\B$-module trace $\trace_{\mathcal{C}_{gh}}$ on $\mathcal{C}_{gh}$ or is the opposite of it.
  This gives us a $2$-cocycle,
  $sgn_{tr}:G\times G\rightarrow\{\pm 1\}$ such that 
  $$\trace_{\mathcal{C}_{g}\boxtimes_{\B}\mathcal{C}_{h}}=sgn_{tr}(g,h)\trace_{\mathcal{C}_{gh}}\ \ \ \ \ \mbox{for all}\ g,h\in G.$$
 Using this $2$-cocycle we obtain a central extension 
 \begin{equation}\label{eq:exten}
 1\rightarrow\Z/2\Z\rightarrow\widetilde{G}\rightarrow G\rightarrow 1
 \end{equation}of $G$ by $\Z/2\Z$.
 Now consider the map of $1$-groups 
 $$\widetilde{G}\rightarrow G\rightarrow\underline{\Pic}^{\trace}(\B)$$
 induced by (\ref{eq:exten}). 
 Let $\D'=\oplus_{\tilde{g}\in\widetilde{G}}\mathcal{C}_{\tilde{g}}$ denote the corresponding $\widetilde{G}$-graded $\B$-bimodule category and $K(\D')$ denote its complexified Grothendieck ring.
 Note that if $\tilde{g},\tilde{g_{1}}\in\widetilde{G}$ such that $\tilde{g},\tilde{g_{1}}\mapsto g$ in $G$ then there is a canonical $\B$-module equivalences 
 $$\mathcal{C}_{\tilde{g}}\cong\mathcal{C}_{\tilde{g_{1}}}\cong\mathcal{C}_{g}\ \ \ \ \ \mbox{for all}\ g\in G$$
 and the normalized $\B$-module trace on $\mathcal{C}_{\tilde{g}}$ and $\mathcal{C}_{\tilde{g_{1}}}$ might be differ by sign.
 Suppose  $\tilde{g}\in \widetilde{G}$ such that $\tilde{g}\mapsto g$ in $G$ and trace on $\mathcal{C}_{\tilde{g}}$, $\mathcal{C}_{g}$ are equal.
 Then the unnormalized crossed $S$-matrices $S^{g}$ and $S^{\tilde{g}}$ are same.
 \begin{remark}\label{remark: general twisted character}
 We will allow slightly more general $g$-twisted character in this section \emph{i.e.}
we will allow the $g$-twisted character is well defined up to $(2\cdot o(g))^{th}$-root of unity.
 \end{remark}
 
In view of Remark \ref{remark: general twisted character} and using the results in \cite{Modularcatcrossedmatrix},\cite{Oncenters} we have:
\begin{theorem}\begin{enumerate}
    \item
For $M\in P_{g},\ C\in P_{1}^{g}$ the numbers $\frac{S^{g}_{M,C}}{\dim(M)}=\frac{S^{g}_{M,C}}{S^{g}_{1,M}}$ and 
$\frac{S^{g}_{M,C}}{\dim(C)}=\frac{S^{g}_{M,C}}{S_{1,C}}$ are cyclotomic integers.
For $C\in P_{1}^{g}$ the linear functional $\chi_{C}^{g}:K(C_{g})\rightarrow \C$, $[M]\mapsto\frac{S^{g}_{M,C}}{S_{1,C}}$ is the $g$-twisted character associated with $\chi_{C}\in\Sim(K(\B))^{g}$.
\item The categorical dimension $\dim(\B)$ is a totally positive cyclotomic integer.
We have 
$$S^{g}\cdot\overline{S^{g}}^{T}=\overline{S^{g}}^{T}\cdot S^{g}=\dim(\B)\cdot I.$$
For $C\in P_{1}^{g}$, we have $\chi_{C}(\alpha_{\chi_{C}})$ is equal to the tatally positive cyclotomic integer $\frac{\dim(\B)}{\dim^{2}(C)}$.
\end{enumerate}
\end{theorem}
\begin{remark}
Under  the above identification of the twisted characters, the  $2$-cocycle $\phi$ representing a $G$-crossed extension $K(\D)$ of $K(\B)$ changes to  $\phi_{1}=sgn_{tr}\cdot \phi$.
\end{remark}
\begin{corollary}\label{coro:twisted Verlinde ribbon}(Twisted Verlinde formula)
 We continue with the same notations as the Corollary \ref{coro: twisted Verlinde formula}.
Moverover assume that $\B$ is a modular category and  $M_{j}\in P_{m_{j}}$ then 
$$\dim \Hom(\mathbbm{1},\Omega_{g_{1},h_{1}}\otimes\cdots\otimes\Omega_{g_{r},h_{r}}\otimes M_{1}\otimes\cdots \otimes M_{s})$$
$$=(\dim(\B))^{r-1}\sum_{D\in P_{1}^{G^{\circ}}}\frac{(\frac{1}{S_{1,D}})^{s+2r-2}S^{m_{1}}_{M_{1},D}\cdots S^{m_{s}}_{M_{s},D}}{\chi_{D}\circ\phi_{1}(g_{1},h_{1},g_{1}^{-1},h_{1}^{-1},\cdots, m_{1},\cdots, m_{s})}.$$
\end{corollary}

\begin{remark}
We continue using the same notation as  Proposition \ref{Thm:pic crossed algebra}.
Fix a $g\in G$ then we have an action of $g$ on $\D$. 
Let us form a $g$-equivariantization $\D^{g}$ from this action.
The objects of $\D^{g}$ can be thought as pairs $(X,\psi_{X})$, where $X\in \D$ and $\psi_{X}:g(X)\xrightarrow{\cong} X$ and  $\D^{g}=\oplus_{h\in C_{G}(g)}\mathcal{C}_{h}^{g}$.
The Grothendieck ring $K(\B^{g})$ is a commutative ring with identity element $[(\mathbbm{1},id_{\mathbbm{1}})]$ and $K(D^{g})$ is a $C_{G}(g)$-graded ring with trivial component equal to $K(\B^{g})$. 
Let us form a quotient algebras
$$K(\B,g)\coloneqq\frac{K(\B)}{([(\mathbbm{1},\omega)]-\omega[\mathbbm{1},id_{\mathbbm{1}}])}\ \ \mbox{and}\ \ K(\D,g)\coloneqq\frac{K(\D)}{([(\mathbbm{1},\omega)]-\omega[\mathbbm{1},id_{\mathbbm{1}}])}$$
where $\omega$ is a primitive  $o(g)^{th}$ root of unity.
The $\star$-map on $K(\D)$ and complex conjugation on $\C$ induces a $\star$-map on $K(\B,g)$ and $K(\D,g)$.
Furthermore, if $g'\in C_{G}(g)$ then we have an induced action of $g'$ on $K(\D,g)$.
Thus $K(\D,g)$ is a $C_{G}(g)$-crossed Frobenius $\star$-algebra with trivial component equal to $K(\B,g)$.
For more details about these algebras  see \cite{Modularcatcrossedmatrix}.
\end{remark}

\subsection{Relationship with strongly graded algebras}\label{sec:strongly graded algebra}

  Let $G$ be a finite group and let $A=\oplus_{g\in G}R_{g}$
be a $G$-graded ring. Recall that $A$ is called \emph{strongly graded} if the multiplication
map $R_{g}\otimes_{\Z} R_{h}\rightarrow R_{gh}$ is surjective. In this situation we say that $A$ is a strongly $G$-graded extension of $R_{1}$. 
By \cite{Groupgraded} for a strongly
$G$-graded ring $A$ the induced maps $R_{g}\otimes_{R_{1}} R_{h}\rightarrow R_{gh}$ are isomorphisms.
In particular $R_{g}$ is an invertible $R_{1}$-bimodule for any $g\in G$.
Thus a strongly $G$-graded extension defines a group homomorphism $g\mapsto R_{g}$ from $G$ to $\Pic(R_{1})$, the group of isomorphism classes of  invertible $R_{1}$-bimodule.
A group homomorphism $\rho:G\rightarrow \Pic(R)$  is called as \emph{realizable} if $\rho$ comes from  some strongly $G$-graded extension $R$ as defined above. 

Let $\rho:G\rightarrow \Pic(R)$ be a group homomorphism.
For each isomorphism class $\rho(g)$, choose an invertible $R$-bimodule $R_{g}$. 
Since $\rho$ is a group homomorphism, the modules $R_{g}\otimes_{R_{1}} R_{h}$ and $R_{gh}$ are isomorphic as $R$-bimodule for each pair $g,h\in G$.
Then we can choose $R$-bimodule isomorphisms
$\phi_{g,h}:R_{g}\otimes_{R_{1}} R_{h}\rightarrow R_{gh}$ with
$\phi_{1,g}(r\otimes x)=rx$ and
$\phi_{g,1}(x\otimes r)=xr$ for $g\in G,\ r\in R,\ x\in R_{g}$.
Then by \cite{Obstruction} the following diagram:
\begin{center}
\begin{tikzcd}
R_{f}\otimes_{R_{1}} R_{g}\otimes_{R_{1}} R_{h}\arrow[r, "\phi_{f,g}\boxtimes Id_{R_{h}}"]\arrow[d, swap, "Id_{R_{f}}\boxtimes \phi_{g,h}"]&
R_{fg}\otimes_{R_{1}} R_{h}\arrow[d, "\phi_{fg,h}"] \\
R_{f}\otimes_{R_{1}} R_{gh}\arrow[r, "\phi_{f,gh}"]&
R_{fgh}
\end{tikzcd}
\end{center}
commutes upto an element $T(\rho)\in H^{3}(G,Z(R)^{\times})$.

The following result is proved in \cite{Obstruction}:
\begin{theorem}\label{realizable}
 The set of isomorphism classes of strongly $G$-graded extensions of ring $R$ corresponding to a group homomorphism $\rho:G\rightarrow \Pic(R)$ is non empty iff $T(\rho)$ is trivial in $ H^{3}(G,Z(R)^{\times})$.
Moreover, the set of isomorphism classes of strongly $G$-graded extensions of ring $R$ corresponding to a group homomorphism $\rho:G\rightarrow \Pic(R)$ is torsor under the abelian group $H^{2}(G,Z(R)^{\times})$.
\end{theorem}
\begin{remark}\label{Remark:automorphism inverible module}
Let $R$ be a ring with unity and let $\alpha$ be an automorphism of  $R$ then one can associate to it an invertible $R$-module $R_{\alpha}$, which is the same left $R$-module as $R$ with right action given by $x\cdot y=x\alpha(y)$ for all $x,y\in R$.
\end{remark}
\begin{lemma}\label{lemma: auto to inver}
Let $\rho:G\rightarrow \mbox{Aut}(R)$ be a group homomorphism.
Then it induces a homomorphism $\rho':G\rightarrow \Pic(R)$ and 
 the corresponding element $T(\rho')\in H^{3}(G,Z(R)^{\times})$ is trivial.
\end{lemma}
\begin{proof}
For any automorphism $\rho_{g}\coloneqq\rho(g)$ of $R$, we can associate  an invertible $R$-module $R_{g}$ as in Remark \ref{Remark:automorphism inverible module}.
Then   $\phi'_{g,h}:R_{g}\otimes_{R} R_{h}\rightarrow R_{gh}$ defined by $\phi'_{g,h}(x\otimes y)=x\rho_{g}(y)$ and extend linearly to $R_{g}\otimes_{R}R_{h}$ is an isomorphism.
Hence, $g\mapsto R_{g}$ is a group homomorphism and $A=\oplus_{g\in G}R_{g}$ is a strongly graded extension of $R$.
By Theorem \ref{realizable}, $T(\rho')\in H^{3}(G,Z(R)^{\times})$ is trivial. 
\end{proof}
\begin{remark}\label{remark:strongly exten and coho bijection}
Let $\rho:G\rightarrow \mbox{Aut}(R)$ be a group homomorphism and let $A=\oplus_{g\in G}R_{g}$ be the corresponding strongly $G$-graded extension of $R$ (constructed in the proof of Lemma \ref{lemma: auto to inver}). 
Using this strongly $G$-graded extension $A$ of $R$, we get a bijection between the set of isomorphism classes of strongly $G$-graded extensions of $R$ corresponding to homomorphism $\rho'$ and $H^{2}(G,Z(R)^{\times})$.
\end{remark}
\begin{remark}
Let $R$ be a commutative Frobenius $\star$-algebra.
Then by \cite[Prop.5.4]{k-theory}, $\mbox{Aut}(R)\ni\alpha\mapsto [R_{\alpha}]\in\Pic(R)$ is an isomorphism. 
\end{remark}

\begin{prop}\label{prop: strongly graded -strict}
Let $R$ be a commutative Frobenius $\star$-algebra.  Suppose $\rho:G\rightarrow\mbox{Aut}^{\Frob}(R)\subset\mbox{Aut}(R)\cong\Pic(R)$ is a group homomorphism.  
Let $A=\oplus_{g\in G}A_{g}$ be a  strongly $G$-graded extension of $R$ corresponding to the homomorphism $\rho$.
Then $A$ has a unique structure of a $G$-graded Frobenius $\star$-extension of $R$ and $Z_{R}(A)$ is a strict $G$-crossed extension of $R$.
\end{prop}
\begin{proof}
Note that $A$ is a strongly $G$-graded extension of $R$ and  $A_{g}$ is isomorphic to $R$ as left  $R$-module for all $g\in G$. 
Also the right action of $R$ on $A_{g}$ is given by $\rho_{g}$.
As $A$ is a strongly $G$-graded extension of $R$, so we have a short exact sequence:
$$1\rightarrow R^{\times}\rightarrow A^{\times}\rightarrow G\rightarrow 1.$$
By Remark \ref{remark: unitary and unit}, every extension of $G$ by $R^{\times}$
 comes from some extension of $G$ by $U$ as in the following commutative diagram:
\begin{center}
\begin{tikzcd}
1\arrow[r]&
U\arrow[hookrightarrow]{r}\arrow[hookrightarrow]{d}&
U(A)\arrow[r]\arrow[hookrightarrow]{d}&
G\arrow[r]\arrow[equal]{d}&
1\\
1\arrow[r]&
R^{\times}\arrow[hookrightarrow]{r}&
A^{\times}\arrow[r]&
G\arrow[r]&
1
\end{tikzcd}
\end{center}
Define a $\star$-map on $U(A)\subseteq A$ as $x^{\star}=x^{-1}$.
Then one can extend this $\star$-map to $A$ by $R$-semilinearly.
Thus, $A$ with the above defined $\star$-map becomes a $G$-graded Frobenius $\star$-algebra. 
This proves the first part of  proposition.

Now, $Z_{R}(A)$ is a $G$-graded Frobenius $\star$-algebra.
Note that $A^{\times}$ acts on $Z_{R}(A)$ by conjugation and the normal subgroup  $R^{\times}$ acts trivially on $Z_{R}(A)$.
Hence we get a well defined action of $G\  (\cong A^{\times}/R^{\times})$ on $Z_{R}(A)$.
With this action of $G$, $Z_{R}(A)$ becomes a $G$-crossed Frobenius $\star$-algebra extension of $R$.
As $A$ is a strongly $G$-graded algebra, the partial action of $G$ on $\Sim(R)$ is the proper action of $G$ on $\Sim(R)$ induced by the given action of $G$ on $R$. In particular, the fixed points for the partial action and the proper action of $G$ on $\Sim(R)$ are equal. Moreover by Remark \ref{decompo:centre of Frob star algebra}(1) the fixed points of the two partial actions on $\Sim(R)$ coming from the $G$-graded algebras $A$ and $Z_R(A)$ are equal.
Hence by the definition of strict $G$-crossed extensions,
$Z_R(A)$ is a strict $G$-crossed Frobenius $\star$-extension of $R$.
This proves the result.
\end{proof}
\begin{remark}
Let $R$ be a commutative Frobenius $\star$-algebra with a group homomorphism $\rho:G\rightarrow \mbox{Aut}^{\Frob}(R)$.
Then by using  Remark \ref{remark:strongly exten and coho bijection}, Theorem \ref{main result} and by Prop. \ref{prop: strongly graded -strict},
 $$A\rightsquigarrow Z_{R}(A)$$
 is a bijection between the set of isomorphism classes of strongly graded extensions of the commutative Frobenius $\star$-algebra $R$ corresponding to the homomorphism $\rho$
and set of all isomorphism classes strict $G$-crossed extensions of $R$ (corresponding to homomorphism $\rho$).
Note that by \cite{Obstruction}, the strongly graded extensions of $R$ as above are also classified by $H^2(G,R^\times)$.
\end{remark}

\bibliographystyle{alpha}
\bibliography{References}
\end{document}